\documentclass{amsart}

\usepackage{amssymb}
\usepackage{overpic}
\usepackage{enumitem}   
\usepackage{graphicx} 
\usepackage{amsrefs}
\usepackage{xcolor}
\usepackage[hidelinks]{hyperref}

\graphicspath{{Figures/}} 

\newtheorem{corollary}{Corollary} 
\newtheorem{remark}{Remark} 
\newtheorem{lemma}{Lemma}

\newtheorem{main}{Theorem} 

\begin{document}
	
\title[Monodromic tangential singularities]{Hopf-like bifurcations induced by\\ hysteresis and time-delay near\\ monodromic tangential singularities}


\author[Douglas D. Novaes, Paulo Santana and David J. W. Simpson]{Douglas D. Novaes$^1$, Paulo Santana$^2$, and David J. W. Simpson$^3$}

\address{$^1$Instituto de Matemática, Estatística e Computação Científica, Universidade Estadual de Campinas, Campinas, 13083-859, Brazil.}
\email{ddnovaes@unicamp.br}

\address{$^2$Instituto de Biociências, Letras e Ciências Exatas, Universidade Estadual Paulista, São José do Rio Preto, 15054-000, Brazil.}
\email{paulo.santana@unesp.br}

\address{$^3$School of Mathematical and Computational Sciences, Massey University,	Palmerston North, 4410, New Zealand.}
\email{d.j.w.simpson@massey.ac.nz}



\begin{abstract}
	We investigate planar piecewise-analytic vector fields, focusing on the merged focus and other monodromic tangential singularities when hysteresis or time-delay is incorporated into the switching condition. If the singularity is an asymptotically stable solution of the system with an instantaneous switch, then the introduction of hysteresis or time-delay causes an attracting limit cycle to be formed locally. We derive asymptotic expressions for the size and period of the limit cycle, allowing any degrees of tangency and any order for the smallest non-zero Lyapunov coefficient. We find that the growth rate of the limit cycle differs for hysteresis and time-delay, and differs to that of the related pseudo-Hopf bifurcation.
\end{abstract}

\maketitle

\section{Introduction}

Many physical systems involve distinct modes of evolution.
Examples include mechanical systems with dry friction~\cite{FeGu98},
 climatic models~\cite{PaPa04},
and control systems with on/off switching strategies (most simply thermostats)~\cite{Ts84}. These give rise to dynamics characterized by phases of smooth evolution interrupted by switching events, and as such are \emph{piecewise-smooth dynamical systems}~\cite{Bristol}.

Piecewise-smooth systems are most simply modeled by assuming that the dynamics switch exactly when the switching threshold is reached. In \cite{FilippovBook},  by means of differential inclusion theory, Filippov introduced a convention for defining solutions of such systems, which assumes that the switch occurs instantaneously once trajectories reach the switching threshold. But in practice, and particularly for relay control systems, there is invaluably some lag between when the threshold is reached and the change in motion takes effect~\cite{TimeDelayBook}. Indeed, hysteresis loop and other memory effects are often incorporated into relay control strategies to reduce excessive switching, known as chattering. For example, a thermostat designed to hold the temperature at a level $T_0$, may refrigerate until the temperature drops to $T_0-\mu$ (where $\mu>0$ is a small constant),
and only turn the refrigeration back on when the temperature reaches $T_0+\mu$~\cite{HysteresisBook}. In general, Filippov models are often made more realistic by incorporating small hysteresis or delay.

As the parameters of a system are varied, its dynamics changes qualitatively
at critical values of the parameter, termed \emph{bifurcations}.
In particular, at a \emph{Hopf bifurcation} a stationary solution changes stability, giving rise to a limit cycle. For this to occur, the system needs to be smooth in a neighborhood of the stationary solution. For smooth systems, there are many generalizations and degenerated versions of the Hopf bifurcation, for instance if the associated Jacobian matrix is nilpotent or degenerate, see~\cite{CLJ} and the references therein.

For piecewise-smooth systems, several analogues of the Hopf bifurcation have also been identified. As far as we know, one of the earliest qualitative studies of such a bifurcation is due to Skryabin~\cite{Skr1978}, who computed the first two Lyapunov coefficients of a \emph{merged focus} in a two-dimensional Filippov system (see Figure~\ref{Fig16}). Another pioneering contribution is due to Filippov, who, in Chapter 4 of his book~\cite{FilippovBook}, computed several Lyapunov coefficients for the same type of singularity, referring to it as a \emph{sewed focus}. At such a singularity, the system has a quadratic tangency on each side of the switching line. Following the notation introduced by Novaes and Silva in \cite[Definition 1]{NovSil2021}, the merged or sewed focus corresponds to a {\it $(2k_L,2k_R)$-monodromic tangential singularity} (see Section~\ref{Sec2} for its characterization) in the lowest-order case, namely, $k_L=k_R=1$. Such singularities arise in many models (see, for instance, \cite[Chapter $8$, $\mathsection6$]{AndVitKha1966}) and have been studied extensively (see, for instance, \cite{CGP2001,EFPT2023,NovSil2021,NovSil2022,NovSil2025}). In the more general setting, that higher order tangencies are allowed, it was proved in \cites{NovSil2021,NovSil2022} that the associated half-return maps and evolution-time functions are analytic, and that the index of the first nonzero Lyapunov coefficient $V_i$, whenever it exists, is necessarily even.

\begin{figure}[ht]
	\begin{center}
		\begin{minipage}{4cm}
			\begin{center} 
				\begin{overpic}[height=4cm]{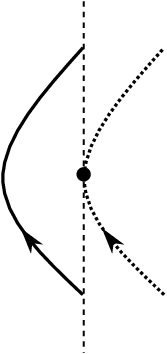} 
				\end{overpic}
				
				Left component.
			\end{center}
		\end{minipage}
		\begin{minipage}{4cm}
			\begin{center} 
				\begin{overpic}[height=4cm]{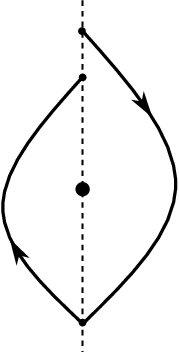} 
				\end{overpic}
				
				Merged focus.
			\end{center}
		\end{minipage}
		\begin{minipage}{4cm}
			\begin{center} 
				\begin{overpic}[height=4cm]{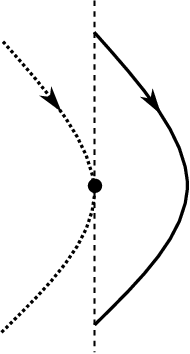} 
				\end{overpic}
				
				Right component
			\end{center}
		\end{minipage}
	\end{center}
	\caption{An illustration of a merged focus. The dashed lines are the switching manifolds. The solid lines are arcs of orbits of the smooth components, while the dotted ones are arcs of orbits not contained in the domain of its respective smooth component.}\label{Fig16}
\end{figure}

A limit cycle can be created by perturbing a system with a monodromic tangential singularity. For example, the perturbation may split the singularity into two distinct tangency points, one for each piece of the system. If the singularity is initially asymptotically stable,
and the split creates a repelling sliding region between the tangency points, then a local attracting limit cycle is born, see \cite[p.~$241$, item $b$]{FilippovBook}. This is known as a \emph{pseudo-Hopf bifurcation}~\cite{KRG2003}. To describe the asymptotics associated with the bifurcation, let $\mu>0$ denote the magnitude of the perturbation. Novaes and Silva, see \cite[Proposition~$4$]{NovSil2025} and~\cite[Corollary~$1$]{NovSil2022}, showed that for any $k_L, k_R \in \mathbb{R}$ the leading-order amplitude and period exponents of the limit cycle is of the form
\begin{equation}\label{eq:amplitude_period}
	\textnormal{amplitude}= c_1\mu^a+O(\mu^{2a}), \quad \textnormal{period}= c_2\mu^b+O(\mu^{2b}),
\end{equation}
with the leading-order amplitude and period exponents given by $a=b=1/(2n)$, where $V_{2n}$ is the first non-zero Lyapunov coefficient of the singularity. See also Arakaki {\em et al.}~\cite{AraNovSan2025} for a generalization to degenerated monodromic singularities, cusps, periodic orbits, and polycycles, for which the leading-order exponents may be negative or given by a logarithm.

By instead perturbing with hysteresis or time-delay, the literature is more scarce. In the lowest-order case $k_L=k_R=1$ with $V_2\neq0$, Li {\em et al.}~\cite{LiYuHan2013} and Kowalczyk~\cite{Kow2017} considered the addition of time-delay for certain classes of systems. This was extended to general piecewise-smooth systems in~\cite[Theorem~$11.4$]{Sim2022} where it was shown that the amplitude and period exponents are $a=b=1/2$ in the lowest-order case of $k_L=k_R=1$ with $V_2\neq0$. With instead hysteresis, Makarenkov~\cite{Mak2017} showed that the exponents are $a=b=1/3$ (still in the case $k_L=k_R=1$ and $V_2\neq0$). This result is notable because the cube-root asymptotics are unique among the $20$ Hopf-like bifurcations of piecewise-smooth systems cataloged in~\cites{Sim2018,Sim2022}.

In this paper we extend these results to arbitrary $(2k_L,2k_R)$-tangential monodromic singularities with  first non-vanishing Lyapunov coefficient $V_{2n}$ under hysteresis and time-delay perturbations. Table~\ref{Table1} summarizes our results on the leading-order terms of the amplitude and period, given by \eqref{eq:amplitude_period}, of the bifurcating limit cycle under hysteresis and time-delay perturbations. For the sake of comparison, the corresponding results for sliding perturbations are also included.
\begin{table}[h]
		\caption{The exponents of the leading-order terms of the amplitude and period, given by~\eqref{eq:amplitude_period}, of the limit cycle bifurcating from a $(2k_L,2k_R)$-monodromic tangential singularity with first non-vanishing Lyapunov coefficient $V_{2n}$ under sliding, hysteresis, and time-delay perturbations.}
		\label{Table1}		
	\begin{tabular}{lcc}
		\hline 
		Perturbation & $a$ & $b$ \\ 
		\hline \vspace{-0.1cm} \\
		Sliding & $\dfrac{1}{2n}$ & $\dfrac{1}{2n}$\vspace{0.2cm} \\
		Hysteresis & $\dfrac{1}{2(n+\max\{k_L,k_R\})-1}$ & $\dfrac{1}{2(n+\max\{k_L,k_R\})-1}$\vspace{0.2cm} \\ 
		Time-delay & $\dfrac{1}{2(n+|k_R-k_L|)}$ & $\dfrac{1}{2(n+|k_R-k_L|)}$ \vspace{0.2cm} \\
		\hline
	\end{tabular}
\end{table}

This reveals a surprising difference between how the limit cycle depends on the orders of the singularity. For hysteresis, $a$ and $b$ are small when one of $k_L$ and $k_R$ is large, while for time-delay, $a$ and $b$ are small when difference between $k_L$ and $k_R$ is large. In contrast, for sliding bifurcations, $a$ and $b$ are independent of $k_L$ and $k_R$. Explicit formulas for the coefficients $c_1$ and $c_2$ in~\eqref{eq:amplitude_period} will be also be provided in our main results.

The paper is organized as follows. In Section~\ref{Sec2}, we state our main results, Theorems~\ref{HystereticHLB} and~\ref{TimeDelayedHLB}, which provide the values of $a$, $b$, $c_1$, and $c_2$ in \eqref{eq:amplitude_period} for the hysteresis and time-delay cases, respectively. In Section~\ref{Sec3} we define and state some technical results regarding half-return maps and their associated evolution times. Proofs of Theorems~\ref{HystereticHLB} and~\ref{TimeDelayedHLB} are provided in Section~\ref{Sec4}, while proofs of the technical lemmas are given in Appendix~\ref{AppA}.

\section{Main results}\label{Sec2}

In this section, we formally introduce the main objects studied in this work, namely tangential monodromic singularities, Lyapunov coefficients, hysteretic systems, and time-delayed systems. We also state our main results concerning these topics.

\subsection{Monodromic tangential singularities}

For each $J\in\{L,R\}$, let $X^J=(f^J,g^J)$ be an analytic planar vector field defined in a neighborhood $U\subset\mathbb{R}^2$ of the origin,
\begin{equation}\label{1}
	\dot x=f^J(x,y), \qquad \dot y=g^J(x,y),
\end{equation}
where
\begin{equation}\label{2}
	f^J(x,y)=\sum_{i,j\ge0}a_{ij}^Jx^iy^j,\qquad
	g^J(x,y)=\sum_{i,j\ge0}b_{ij}^Jx^iy^j,
\end{equation}
with $a_{ij}^J,b_{ij}^J\in\mathbb{R}$.
We consider the piecewise analytic system
\begin{equation}\label{piecewisesystem}
	Z(x,y)=
		\begin{cases}
			X^L(x,y), & x>0,\\
			X^R(x,y), & x<0,
		\end{cases}
\end{equation}
with switching manifold $\Sigma=\{(x,y)\in U\colon x=0\}.$

Let $h(x,y)=x$, and notice $\Sigma=h^{-1}(\{0\})$. Let $\Phi^J(t,x,y)$ be the solution to $X^J$ satisfying $\Phi^J(0,x,y)=(x,y)$. Then
\begin{equation}\label{eq:hOfPhiJ}
	h\left(\Phi^J(t,x,y)\right) = h(x,y)+\sum_{i\geqslant1}\frac{(X^J)^i h(x,y)}{i!}t^i,
\end{equation}
where $(X^J)^i h(x,y)$ is the \emph{Lie derivative} of order $i$ of $h$
in the direction of $X^J$ evaluated at $(x,y)$. The first Lie derivative is given by
\[
	X^J h = \left<\nabla h,X^J \right>,
\]
with $\left<\cdot,\cdot\right>$ denoting the standard inner product on $\mathbb{R}^2$. Higher order Lie derivatives can be evaluated iteratively via
\[
	(X^J)^m h = \left<\nabla (X^J)^{m-1} h, X^J \right>.
\]

Given $m\geqslant1$, we say that $p=(x,y)\in\Sigma$ has
\emph{contact of multiplicity} $m$ with $X^J$ if $t=0$ is a root of multiplicity $m$ of $h\left( \Phi^J(t,p) \right)$. In view of~\eqref{eq:hOfPhiJ}, this is equivalent to 
\begin{equation}\label{eq:XJhp}
	X^Jh(p)=(X^J)^2h(p)=\ldots=(X^J)^{m-1}h(p)=0, \quad (X^J)^mh(p)\neq0.
\end{equation}
Odd contact multiplicities do not yield monodromic behavior or local limit cycles, so we only treat even contact multiplicities, which can be divided into two types. We say that a contact point $p$ of multiplicity $2k_R$ (resp.~$2k_L$) of $X^R$ (resp.~$X^L$) is \emph{invisible} if $(X^R)^{2k_R}h(p)<0$ (resp.~$(X^L)^{2k_L}h(p)>0$), and \emph{visible} otherwise, see Figure~\ref{Fig2}.
\begin{figure}[ht]
	\begin{center}
		\begin{minipage}{6cm}
			\begin{center} 
				\begin{overpic}[height=5cm]{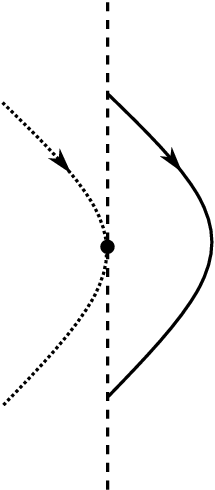} 
					\put(24.5,49){$p$}
					\put(24.5,98){$y=0$}				
				\end{overpic}
				
				Invisible
			\end{center}
		\end{minipage}
		\begin{minipage}{6cm}
			\begin{center} 
				\begin{overpic}[height=5cm]{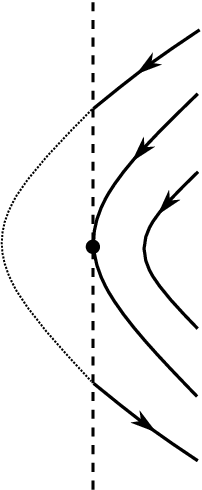} 
					\put(13,49){$p$}
					\put(21,98){$y=0$}
				\end{overpic}
				
				Visible
			\end{center}
		\end{minipage}
	\end{center}
\caption{An illustration of invisible and visible contact points of $X^R$.
For an invisible contact point, the orbit through $p$ intersects $\Sigma^R=\{(x,y)\in U\colon x\geqslant0\}$
only at $p$ (and thus is `invisible'). For a visible contact point, the orbit is contained in $\Sigma^R$ (and thus is `visible').}\label{Fig2}
\end{figure}

Following the notation introduced in \cite[Definition 1]{NovSil2021}, a tangential point $p\in\Sigma$ is called a {\it $(2k_L,2k_R)$-monodromic tangential singularity} of $Z$ if it is an invisible contact point for both vector fields $X^L$ and $X^R$, of multiplicities $2k_L$ and $2k_R$, respectively, and if $Z$ admits a first-return map defined in a neighborhood of $p$. The latter condition means that $Z$ is monodromic around $p$ and, in this case, is equivalent to requiring that $X^L(p)$ and $X^R(p)$ be anti-collinear, that is, $X^L(p)=\lambda X^R(p),$ for some $\lambda<0.$ This ensures that the orientations of the trajectories of $Z^+$ and $Z^-$ near $p$ are compatible, so that trajectories can be continued across $\Sigma$ by alternating the flows of $X^L$ and $X^R$.

In the particular case where $p=(0,0)$, the definition of a $(2k_L,2k_R)$-monodromic tangential singularity for system \eqref{piecewisesystem} can be readily reformulated as the following characterization:
\begin{lemma}\label{LC}
	The origin is a $(2k_L,2k_R)$-monodromic tangential singularity of the piecewise analytic system \eqref{piecewisesystem} if and only if
	\begin{enumerate}[label=(\roman*)]
		\item $a_{0,i}^J=0$ for all $i\in\{0,\dots,2k_J-2\}$ and $J \in \{ L, R \}$,
		\item $b_{00}^R a_{0,2k_R-1}^R<0$, $b_{00}^L a_{0,2k_L-1}^L>0$, and
		\item $b_{00}^Lb_{00}^R<0$.
	\end{enumerate}
\end{lemma}
The proof of Lemma~\ref{LC} follows from a straightforward computation of the Lie derivatives and is therefore omitted. Indeed, condition~(i) is equivalent to the first part of~\eqref{eq:XJhp}, condition~(ii) ensures that the origin is an invisible contact point for both $X^L$ and $X^R$, and condition~(iii) corresponds to the anti-collinearity of $X^L(0)$ and $X^R(0)$.

\subsection{Lyapunov coefficients}

Write $\Phi^J(t,x,y)=\big(\varphi^J(t,x,y),\psi^J(t,x,y)\big)$ for the components of $\Phi^J$. It was proved in~\cite{NovSil2021} that, if the origin is an invisible contact point of $X^J$, then there exists $\varepsilon>0$ such that the half-return map $P_0^J:(-\varepsilon,\varepsilon)\to\mathbb{R}$ is well defined on $\Sigma$ near the origin, describing the evolution of trajectories from and back to the switching manifold. More precisely, there exists a unique function $T_0^J:(-\varepsilon,\varepsilon)\to\mathbb{R}$ with $T_0^J(0)=0$ satisfying
\begin{equation}\label{50}
	\varphi^J(t,0,y)=0 \iff t=T_0^J(y),
\end{equation}
so that
\begin{equation}\label{51}
	P_0^J(y)=\psi^J\big(T_0^J(y),0,y\big).
\end{equation}
Moreover, $P_0^J$ is analytic and admits an expansion of the form
\begin{equation}
	P_0^J(y)=-y+\sum_{i\geqslant2}\alpha_i^J y^i.
\end{equation}
It was further shown in~\cite{NovSil2022} that $T_0^J$ is analytic and has an expansion of the form
\begin{equation}\label{49T}
	T_0^J(y)=-\frac{2}{b_{00}^J}y+\sum_{i\geqslant2}\gamma_i^J y^i.
\end{equation}
Recursive formulae for computing the coefficients $\alpha_i^J,\gamma_i^J\in\mathbb{R}$ are provided in~\cite{NovSil2021,NovSil2022}.

Now suppose the origin is a monodromic tangential singularity of $Z$.
Let $\delta = -1$ if solutions revolve clockwise,
and $\delta = 1$ if solutions revolve counterclockwise; equivalently
\begin{equation}\label{eq:delta}
	\delta = \operatorname{sgn} b^R_{00} \,.
\end{equation}

In order to study pseudo-Hopf bifurcations, the authors in~\cite{NovSil2021,NovSil2025} considered the displacement map
\begin{equation}\label{13}
	D_0(y):=\delta\big(P^R_0(y)-P^L_0(y)\big)=\sum_{i\geqslant2}V_iy^i,
\end{equation}
with $V_i=\delta(\alpha_i^R-\alpha_i^L)$ known as the $i$-th \emph{Lyapunov coefficient}. The factor $\delta$ in~\eqref{13} is included so that $D_0(y)>0$ if the origin is repelling, and $D_0(y)<0$ if the origin is attracting, see Figure~\ref{Fig9}. 
\begin{figure}[ht]
	\begin{center}
		\begin{minipage}{6cm}
			\begin{center} 
				\begin{overpic}[height=5cm]{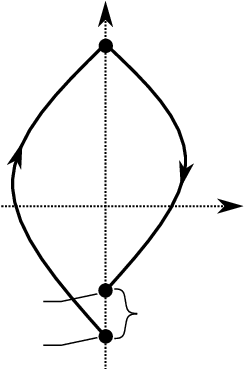} 
					\put(63,39){$x$}
					\put(31,98){$y$}
					\put(23,88){$p$}
					\put(3.5,5.5){$q_L$}
					\put(3.5,17.5){$q_R$}
					\put(37.5,12.5){$D_0(p)<0$}
				\end{overpic}
				
				$(a)$ $\delta=-1$.
			\end{center}
		\end{minipage}
		\begin{minipage}{6cm}
			\begin{center} 
				\begin{overpic}[height=5cm]{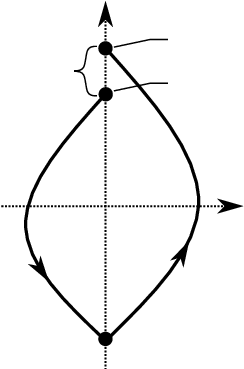} 
					\put(63,39){$x$}
					\put(31,98){$y$}
					\put(23,3.5){$p$}
					\put(46.5,76){$q_L$}
					\put(46.5,88){$q_R$}
					\put(-12,78.75){$D_0(p)>0$}
				\end{overpic}
				
				$(b)$ $\delta=1$.
			\end{center}
		\end{minipage}
	\end{center}
\caption{An illustration of the displacement map $D_0$ for a point $p = (0,y)$, where $q_L = \left( 0, P^L_0(y) \right)$ and $q_R = \left( 0, P^R_0(y) \right)$ are the two return points. Notice that with either $y > 0$ and $\delta = -1$ (clockwise motion in (a)) or $y < 0$ and $\delta = 1$ (counterclockwise motion in (b)), the evolution time $T^L_0(y)$ is negative
and the evolution time $T^R_0(y)$ is positive.}\label{Fig9}
\end{figure}
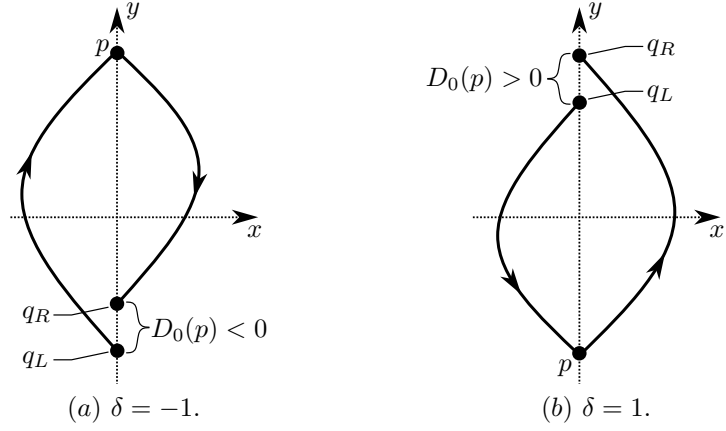
We remark from~\cite[Theorem~$B$]{NovSil2021} that if $V_i\neq0$ for some $i\geqslant2$, then the minimal such index $i$ is even. From~\cite[Corollary~$1$]{NovSil2021}, $V_2=\delta(\alpha_2^R-\alpha_2^L)$ where
\begin{equation}\label{eq:alpha2J}
	\alpha_2^J=\frac{2}{2k_J+1}\left(\frac{a_{10}^J}{b_{00}^J}+\frac{b_{01}^J}{b_{00}^J}-\frac{a_{0,2k_J}^J}{a_{0,2k_J-1}^J}\right),
\end{equation}
for each $J\in\{L,R\}$, and all other Lyapunov coefficients can be evaluated via the algorithm given in~\cite[Theorem~$C$]{NovSil2021}.

\subsection{Hysteretic Hopf-like bifurcation}

Given $\mu>0$, we associate to $Z$ the \emph{hysteretic system}~$Z_h$ defined as follows. First, let
\[
	\Sigma^L_\mu=\{(x,y)\in\mathbb{R}^2\colon x\leqslant-\mu\}, \quad \Sigma^R_\mu=\{(x,y)\in\mathbb{R}^2\colon x\geqslant\mu\}.
\]
The forward orbit of a point $p_1\in\Sigma^R_\mu$ is defined by $X^R$ until the orbit reaches the line $x=-\mu$, at a point $p_2$. At this instant the system switches and the orbit is governed by $X^L$ until it reaches $x=\mu$ at a point $p_3$, where it switches again to $X^R$. This is illustrated in Figure~\ref{Fig1} in the case of a point $p_1$ on the line $x=\mu$. 
\begin{figure}[ht]
	\begin{center}
		\begin{minipage}{6cm}
			\begin{center} 
				\begin{overpic}[width=4.5cm]{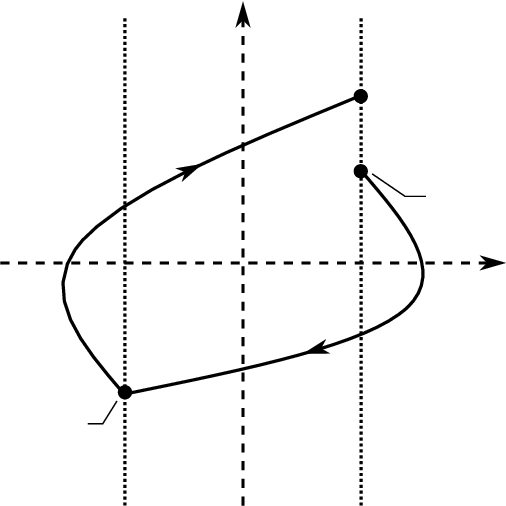} 
					\put(85,60){$p_1$}
					\put(9,15){$p_2$}
					\put(74,80){$p_3$}
					\put(97,42){$x$}
					\put(50,98){$y$}
					\put(4,95){{\tiny$x=-\mu$}}
					\put(73,95){{\tiny$x=\mu$}}
				\end{overpic}
				
				$\delta=-1$.
			\end{center}
		\end{minipage}
		\begin{minipage}{6cm}
			\begin{center} 
				\begin{overpic}[width=4.5cm]{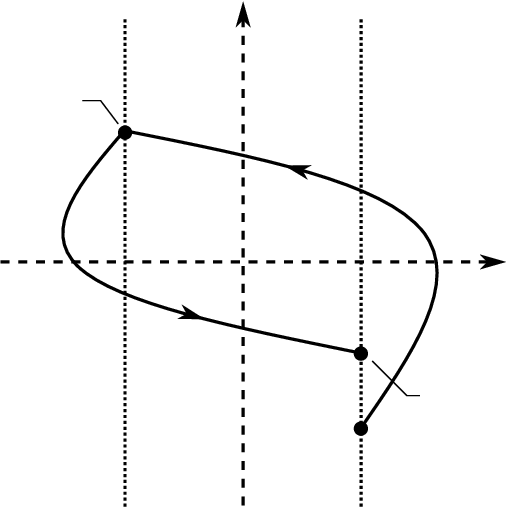} 
					\put(74,13){$p_1$}
					\put(8,79){$p_2$}
					\put(84,20.5){$p_3$}
					\put(97,43){$x$}
					\put(50,98){$y$}
					\put(4,95){{\tiny$x=-\mu$}}
					\put(73,95){{\tiny$x=\mu$}}
				\end{overpic}
				
				$\delta=1$.
			\end{center}
		\end{minipage}
	\end{center}
\caption{Parts of typical orbits of hysteretic systems having a
$(2,2)$-monodromic tangential singularity at the origin.}\label{Fig1}
\end{figure}
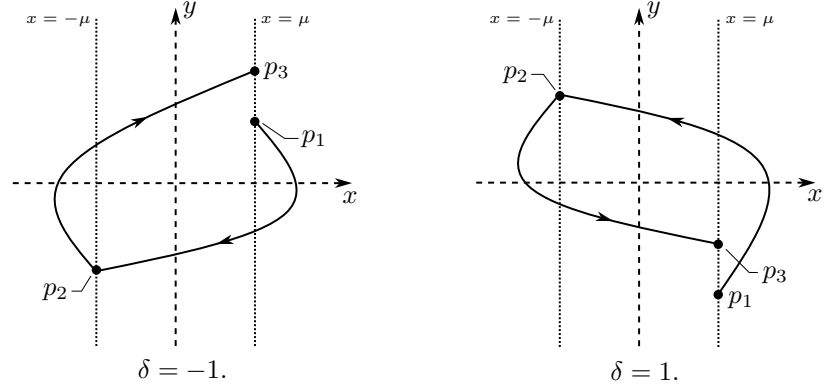
If we start with $p_1\in\Sigma^L_\mu$, the dynamics follows similarly but with starting motion given by $X^L$. Any other initial point belongs to the set
\[
	\Sigma^N:=\mathbb{R}^2\setminus\big(\Sigma^R_\mu\cup\Sigma^L_\mu\big)=\{(x,y)\in\mathbb{R}^2\colon -\mu<x<\mu\},
\]
from which forward motion is not uniquely determined and depends on the preceding states. The set $\Sigma^N$ is referred to as the \emph{non-unique zone} by Andronov \emph{et al.}~\cite{AndVitKha1966}.

For simplicity, we write the hysteretic system $Z_h$ as
\begin{equation}\label{3}
	\left(\begin{array}{c}
		\dot x \\
		\dot y 
	\end{array}\right)
	=
	\left\{\begin{array}{ll}
		X^L(x,y) & \text{until } x=\mu, \\
		X^R(x,y) & \text{until } x=-\mu.
	\end{array}\right.		
\end{equation}
Notice that as $\mu\to0^+$ the dynamics of $Z_h$ collapses to the dynamics of $Z$, provided no sliding segments appear.

Now suppose $Z$ has a $(2k_L,2k_R)$-monodromic tangential singularity at the origin, and let
\begin{equation}\label{57}
	\eta^J_h=\frac{2b_{00}^J}{a_{0,2k_J-1}^J}, 
	\quad
	\kappa_h=\left\{\begin{array}{ll}
				-\eta^L_h, &\text{if } k_L>k_R, 
				\vspace{0.1cm} \\
				\eta^R_h-\eta^L_h, &\text{if } k_L=k_R, 
				\vspace{0.1cm} \\
				\eta^R_h, &\text{if } k_L<k_R.
	\end{array}\right.
\end{equation}
We remark from Lemma~\ref{LC} that $\eta^R_h<0$, $\eta^L_h>0$, and thus $\kappa_h<0$. Also let $k_M=\max\{k_L,k_R\}$, and
\begin{equation*}
	S_h^{\mu,\varepsilon}=\left\{(x,y)\in\{\mu\}\times(-\varepsilon,\varepsilon)\colon |y|^{2k_M+\frac{1}{2}}>\mu, \, \delta y<0\right\},
\end{equation*}
shown bellow in Figure~\ref{Fig17}. The bifurcation of a limit cycle from a $(2k_L,2k_R)$-monodromic tangential singularity for the hysteretic system~\eqref{3} was studied in~\cite[Theorem~$2.9$]{Mak2017} and~\cite[Theorem~$11.3$]{Sim2022} in the simplest case that $k_L=k_R=1$ and $V_2\neq0$. The following theorem allows any $k_L,k_R\geqslant1$, and any $n \geqslant 1$ such that $V_{2n}$ is the first non-zero Lyapunov coefficient.

\begin{main}\label{HystereticHLB}
	Suppose that the origin is a $(2k_L,2k_R)$-monodromic tangential singularity and that there is a minimal $n\in\mathbb{N}$ such that $V_{2n}\neq0$. If $V_{2n}<0$ (resp.~$V_{2n}>0$), then there exist $\varepsilon,\overline\mu>0$ such that, for every $\mu\in (0,\overline\mu)$, the hysteretic system \eqref{3} has a unique periodic orbit (resp.~no periodic orbit) intersecting $S_h^{\mu,\varepsilon}$. In the case that a periodic orbit exists, it is a hyperbolic attracting limit cycle. Moreover, letting $\big(\mu,-\delta\mathcal{P}_h(\mu)\big)$ denote its intersection with~$S_h^{\mu,\varepsilon}$ and $\mathcal{T}_h(\mu)$ its period, 
\begin{align}
		\mathcal{P}_h(\mu) &= \left|\frac{\kappa_h}{V_{2n}}\right|^{1/N} \mu^{1/N}+O\big(\mu^{2/N}\big), \label{45}
		\vspace{0.2cm} \\
		\mathcal{T}_h(\mu) &= \left|\frac{2}{b_{00}^L}-\frac{2}{b_{00}^R}\right| \left|\frac{\kappa_h}{V_{2n}}\right|^{1/N} \mu^{1/N}+O\big(\mu^{2/N}\big), \label{46}
	\end{align}
	where $N=2(n+k_M)-1$.
\end{main}

\begin{remark}\label{rem:onlyattracting}
In contrast to the pseudo-Hopf bifurcation, where a repelling limit cycle may bifurcate when $V_{2n}>0$, Theorem~\ref{HystereticHLB} shows that, in the hysteretic case, only an attracting limit cycle can arise. This occurs because the hysteresis effectively converts the discontinuity line into a repelling strip that mimics the effect of an unstable sliding set (also called escaping region in some works). By contrast, in the pseudo-Hopf case, one may also obtain a stable sliding region, which is responsible for the bifurcation of a repelling limit cycle when $V_{2n}>0$. Such a mechanism is absent in the hysteretic case.
\end{remark}

As a simple illustration of Theorem~\ref{HystereticHLB}, suppose
\begin{equation}\label{Example}
	X^L(x,y)=(y-x,1), \quad X^R=(y^3,-1-y).
\end{equation}
In this case, $a^L_{01}=1$, $a^L_{10}=-1$, $b^L_{00}=1$, $a^R_{03}=1$, $b^R_{00}=-1$, $b^R_{01}=-1$, and all other coefficients are zero. Thus, by Lemma~\ref{LC}, the origin is a $(2,4)$-monodromic tangential singularity. By evaluating~\eqref{eq:alpha2J}, we obtain $\alpha^L_2=-\frac{2}{3}$ and $\alpha^R_2=\frac{2}{5}$. Also $\delta=-1$ by~\eqref{eq:delta}, so $V_2=-\frac{16}{15}$. Thus we can apply Theorem~\ref{HystereticHLB} with $k_L=1$, $k_R=2$, and $n=1$. The theorem shows there exists a unique attracting limit cycle for sufficiently small $\mu>0$ with amplitude and period asymptotically proportional to $\mu^\frac{1}{N}$, where $N=5$.

The solid curve in Figure \ref{Fig14}(a) shows the position of the limit cycle as a function of $\mu$, obtained by numeric calculations. The dashed curve shows the leading order term of~\eqref{45},
and as expected the two curves agree as $\mu\to0$. Figure~\ref{Fig14}(b) shows the limit cycle in phase space for $\mu=0.004$.
\begin{figure}[ht]
	\begin{center}
		\begin{minipage}{6cm}
			\begin{center} 
				\begin{overpic}[height=5.5cm]{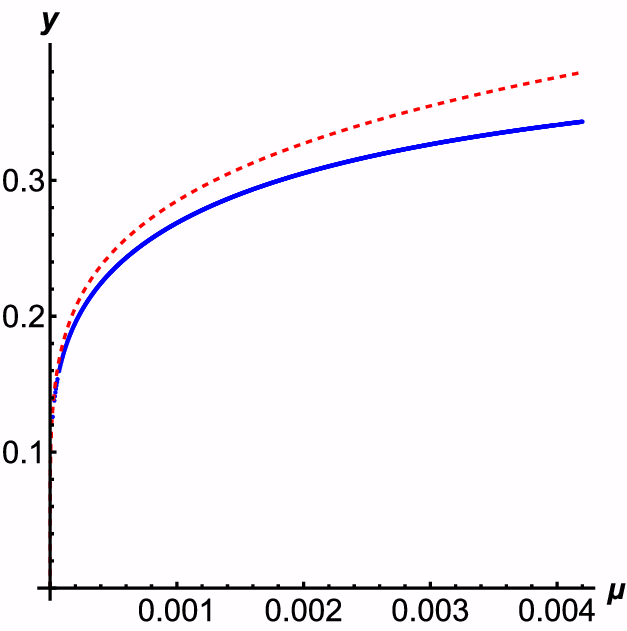} 
				\end{overpic}
				
				$(a)$
			\end{center}
		\end{minipage}
		\begin{minipage}{6cm}
			\begin{center} 
				\begin{overpic}[height=5.5cm]{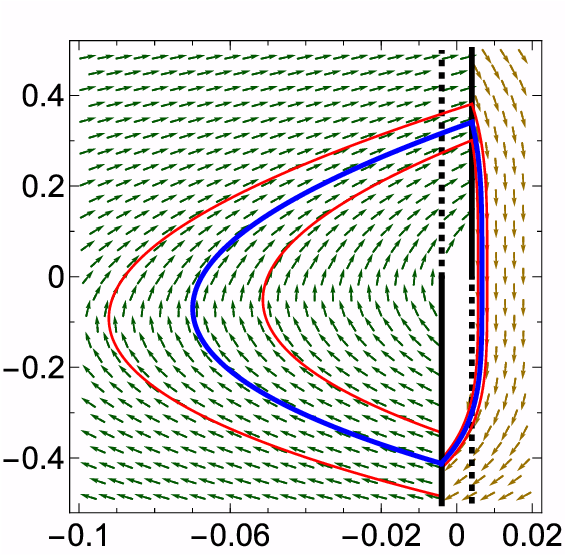} 
				\end{overpic}
				
				$(b)$
			\end{center}
		\end{minipage}
	\end{center}
\caption{A numerical verification of Theorem~\ref{HystereticHLB} for the hysteretic system~\eqref{3} with~\eqref{Example}. The solid curve in~(a) is the position of the limit cycle (computed numerically), while the dashed curve is leading order term in~\eqref{45}. In (b) we show a phase portrait for $\mu=0.004$. The blue loop is the limit cycle, the red curves are typical forward orbits, and the black lines are the switching lines $x=\pm\mu$.}\label{Fig14}
\end{figure}

\subsection{Time-delayed Hopf-like bifurcation}

Given $\mu>0$, we introduce the time-delayed system~$Z_{td}$:
\begin{equation}\label{3x}
	\left(\begin{array}{c}
		\dot x(t) 
		\vspace{0.1cm} \\
		\dot y(t) 
	\end{array}\right)
	=
	\left\{\begin{array}{ll}
		X^L\big(x(t),y(t)\big) & \text{while } x(t-\mu)<0, 
		\vspace{0.1cm} \\
		X^R\big(x(t),y(t)\big) & \text{while } x(t-\mu)>0.
	\end{array}\right.		
\end{equation}
Delay differential equations are infinite-dimensional, but in~\eqref{3x} the time delay~$\mu$ appears only in the switching condition. Consequently, if the time between switches is greater than~$\mu$, then the motion is essentially two-dimensional and can be characterized with one or more one-dimensional return maps.

To clarify, consider an orbit $\Phi(t)=(x(t),y(t))$ of~\eqref{3x}
that at a time $t_1\in\mathbb{R}$ is located at $p_1=\Phi(t_1)=(0,y_1)\in \Sigma$, and for which $\Phi(t)\in\Sigma^L_0$ for all $t_1-\mu<t<t_1$. The orbit follows $X^L$ until arriving at $p_2=\Phi(t_2)$, where $t_2=t_1+\mu$. Suppose $\Phi(t)\in\Sigma^R_0$ for all $t_1<t< t_2$, which in our setting typically occurs when $\delta y_1<0$. Then at $t=t_2$ the orbit switches to following $X^R$, as in Figure~\ref{Fig10}.

Furthermore, suppose the orbit returns to $\Sigma$, and let $t_3>t_2$ be the first time at which this occurs. Then the orbit follows $X^R$ through $p_3=\Phi(t_3)$ until arriving at $p_4=\Phi(t_4)$, where $t_4=t_3+\mu$. If $\Phi(t)\in\Sigma^L_0$ for all $t_3<t<t_4$, the orbit switches to following $X^L$ at $t=t_4$. The orbit may subsequently return to $\Sigma$ at a point $p_5=(0,y_5)$, as in Figure~\ref{Fig10}. In this case the entirety of the motion can be captured with a one-dimensional map that compares the values of $y_1$ and $y_5$, and if $y_1=y_5$ then the orbit forms a periodic orbit, which is also a limit cycle if it is isolated.
\begin{figure}[ht]
	\begin{center}
		\begin{minipage}{6cm}
			\begin{center} 
				\begin{overpic}[height=5cm]{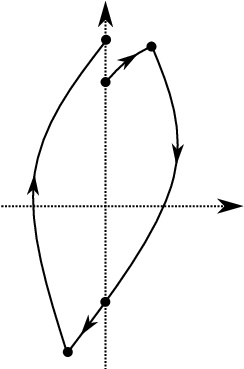} 
					\put(63,39){$x$}
					\put(24,98){$y$}
					\put(22,74){$p_1$}
					\put(43.5,87){$p_2$}
					\put(30.5,16){$p_3$}
					\put(12,0){$p_4$}
					\put(31,88.5){$p_5$}
				\end{overpic}
				
				$\delta=-1$.
			\end{center}
		\end{minipage}
		\begin{minipage}{6cm}
			\begin{center} 
				\begin{overpic}[height=5cm]{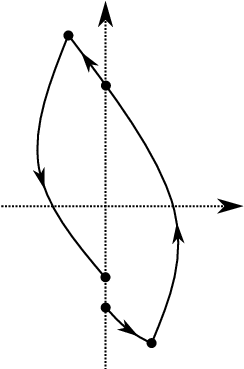} 
					\put(63,39){$x$}
					\put(24,98){$y$}
					\put(21,15.5){$p_1$}
					\put(44,5){$p_2$}
					\put(31,77){$p_3$}
					\put(10,90){$p_4$}
					\put(31,23.5){$p_5$}
				\end{overpic}
				
				$\delta=1$.
			\end{center}
		\end{minipage}
	\end{center}
\caption{Sketches of typical orbits of the time-delayed system~\eqref{3x} when $Z$ has a monodromic tangential singularity at the origin.}\label{Fig10}
\end{figure}
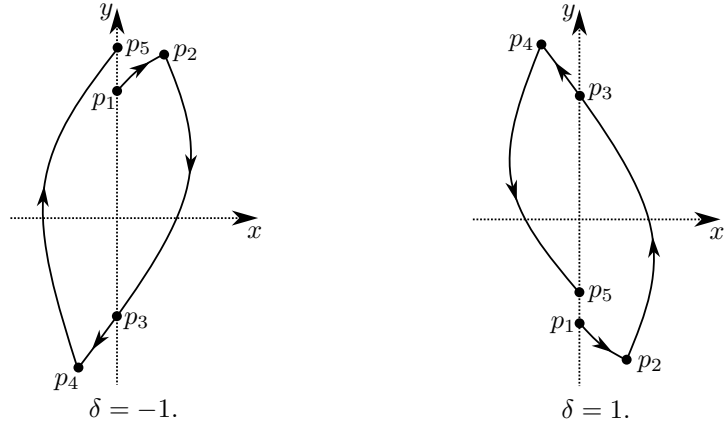

Suppose $Z$ has a $(2k_L,2k_R)$-monodromic tangential singularity at the origin. Given $J, K\in\{L,R\}$ with $J\neq K$, let
\begin{equation}\label{37}
	\eta_{td}^J=\left\{\begin{array}{ll}
		\displaystyle \frac{a_{0,2k_K-1}^K}{a_{0,2k_J-1}^J}b_{00}^J, & \text{if } k_J>k_K, 
		\vspace{0.2cm} \\
		\displaystyle \frac{a_{0,2k_K-1}^K}{a_{0,2k_J-1}^J}b_{00}^J - b_{00}^K, &\text{if } k_J=k_K, 
		\vspace{0.2cm} \\
		-b_{00}^K, &\text{if } k_J<k_K,
	\end{array}\right.
\end{equation}
and
\begin{equation}\label{36}
	\kappa_{td}=\left\{\begin{array}{ll}
		-\eta^L_{td}, &\text{if } k_L>k_R, 
		\vspace{0.1cm} \\
		\eta^R_{td}-\eta^L_{td}, &\text{if } k_L=k_R, 
		\vspace{0.1cm} \\
		\eta^R_{td}, &\text{if } k_L<k_R.
	\end{array}\right.
\end{equation}
It follows from~\eqref{eq:delta} and Lemma~\ref{LC} that in all three cases
\begin{align*}
	\operatorname{sgn} \eta^L_{td} &= -\delta, &
	\operatorname{sgn} \eta^R_{td} &= \delta, &
	\operatorname{sgn} \kappa_{td} &= \delta.
\end{align*}
Also let $a_M=1+2|k_R-k_L|$, and
\begin{equation*}
	S_{td}^{\mu,\varepsilon}=\left\{(x,y)\in\{0\}\times(-\varepsilon,\varepsilon)\colon |y|^{a_M+\frac{1}{2}}>\mu, \, \delta y<0\right\}.
\end{equation*}
In the simplest case that $k_L=k_R=1$ and $V_2\neq0$, the limit cycle created by increasing the value of $\mu$ from zero was analyzed in~\cite{LiYuHan2013,Kow2017} for piecewise-linear systems, and~\cite[Theorem~$11.4$]{Sim2022} for general piecewise-smooth systems. Here we allow any $k_L,k_R\geqslant1$ and any $n\geqslant1$ such that~$V_{2n}$ is the first non-zero Lyapunov coefficient.

\begin{main}\label{TimeDelayedHLB}
	Suppose the origin is a $(2k_L,2k_R)$-monodromic tangential singularity and there is a minimal $n\in\mathbb{N}$ such that $V_{2n}\neq0$. If $V_{2n}<0$ (resp. $V_{2n}>0$), then there exist $\varepsilon,\overline{\mu}>0$ such that, for every $\mu\in(0,\overline{\mu})$, the time-delayed system~\eqref{3x} has a unique periodic orbit (resp. no periodic orbit) intersecting $S_{td}^{\mu,\varepsilon}$. In case of the existence of a periodic orbit, it is a hyperbolic  attracting limit cycle. Moreover, letting $(0,-\delta\mathcal{P}_{td}(\mu))$ denote its intersection with $S_{td}^{\mu,\varepsilon}$ and $\mathcal{T}_{td}(\mu)$ its period, 
	\begin{align}
		\mathcal{P}_{td}(\mu) &= \left|\frac{\kappa_{td}}{V_{2n}}\right|^{1/M} \mu^{1/M}+O\big(\mu^{2/M}\big), \label{45x}
		\vspace{0.2cm} \\
		\mathcal{T}_{td}(\mu) &= \left|\frac{2}{b_{00}^L}-\frac{2}{b_{00}^R}\right| \left|\frac{\kappa_{td}}{V_{2n}}\right|^{1/M} \mu^{1/M}+O\big(\mu^{2/M}\big), \label{46x}
	\end{align}
	where $M=2\big(n+|k_R-k_L|\big)$.
\end{main}

The observation made in Remark~\ref{rem:onlyattracting} for Theorem~\ref{HystereticHLB}, concerning the fact that only an attracting limit cycle can bifurcate, also applies to Theorem~\ref{TimeDelayedHLB}.

To illustrate Theorem \ref{TimeDelayedHLB}, Figure~\ref{Fig15} shows a bifurcation diagram and sample phase portrait of the time-delayed system~\eqref{3x} with vector fields~\eqref{Example}. The solid curve shows the position of the limit cycle as a function of $\mu$, obtained by numeric calculations. The dashed curve shows the leading order term of~\eqref{45x}, which has $M=4$ and as expected the two curves agree as $\mu\to0$. Figure~\ref{Fig14}(b) shows the limit cycle in phase space for $\mu=0.03$.
\begin{figure}[ht]
	\begin{center}
		\begin{minipage}{6cm}
			\begin{center} 
				\begin{overpic}[height=5.5cm]{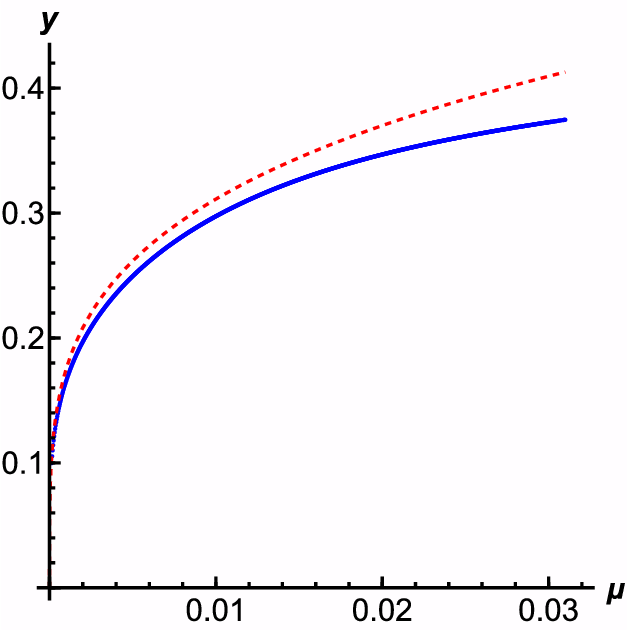} 
				\end{overpic}
				
				$(a)$
			\end{center}
		\end{minipage}
		\begin{minipage}{6cm}
			\begin{center} 
				\begin{overpic}[height=5.5cm]{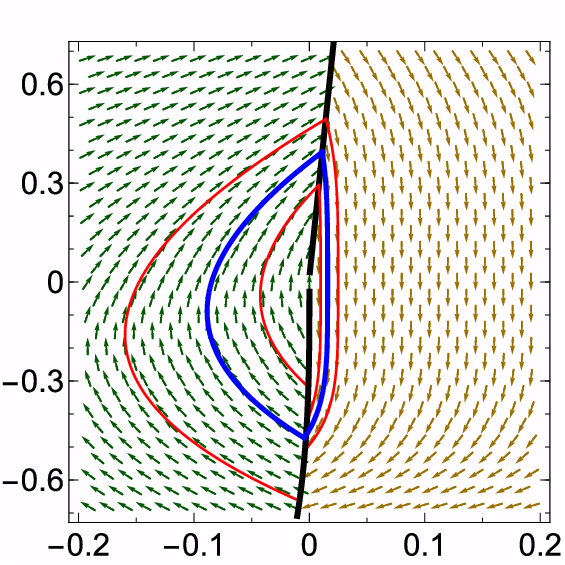} 
				\end{overpic}
				
				$(b)$
			\end{center}
		\end{minipage}
	\end{center}
\caption{A numerical verification of Theorem~\ref{TimeDelayedHLB} for the time-delayed system~\eqref{3x} with~\eqref{Example}. The solid curve in $(a)$ shows the position of the limit cycle (computed numerically), while the dashed curve is the leading order term in~\eqref{45x}. In $(b)$ we show a phase portrait for $\mu=0.03$ with the limit cycle in blue. The black curves are where orbits switch having traveled past $\Sigma$ by a time $\mu$.}\label{Fig15}
\end{figure}

\section{Existence and properties of half-return maps}\label{Sec3}

In this section we extend the return time function $T_0^J$ and half-return map $P_0^J$ studied in~\cite{NovSil2021,NovSil2022} to small $\mu>0$, generalizing the relations~\eqref{50} and~\eqref{51} to hysteresis and time-delay.

\subsection{Hysteresis}\label{Sec3.1}

Recall, the half-maps $P_0^R$ and $P_0^L$ correspond to evolution
to and from the line $x=0$ (i.e.~the switching manifold $\Sigma$). In the case $\delta =-1$, orbits rotate clockwise around the origin, and consider a point $(0,y)$ with $y>0$. Then $P_0^R(y)$ is defined by following the forward orbit of $(0,y)$ under the right piece of the system until it reintersects $x=0$, and $P_0^L(y)$ is defined by following the backward orbit of $(0,y)$ under the left piece of the system until it reintersects $x=0$, see Figure~\ref{Fig9}$(a)$.

For the hysteretic system~\eqref{3}, we instead work with the lines $x=\mu$ and $x=-\mu$, as these lines are where orbits switch. However, since we allow orbits to have quartic or higher-order tangencies with the switching manifold, orbits can have many intersections with $x=\pm\mu$ in a neighborhood of $(\mu,y)=(0,0)$. For this reason, we introduce the sets
\begin{align}
	A_h^J &= \left\{(x,y)\in(-\varepsilon,\varepsilon)^2\colon|x|<|y|^{2k_J+\frac{1}{2}}\right\}, \label{48A} \\
	W_h^J &= \left\{(\mu,y)\in(0,\varepsilon)\times(-\varepsilon,\varepsilon)\colon \mu<|y|^{2k_J+\frac{1}{2}}, \, \delta y<0\right\}, \label{48W}
\end{align}
for $J\in\{L,R\}$, where $\varepsilon>0$ is suitably small.
As we shall see in Appendix~\ref{AppA}, by taking pairs $(\mu,y) \in W^J_h$, the value $|y|$ is large enough that the orbit of $X^J$ through $(\mu,y)$ only has two intersections with $x=-\mu$ in $A_h^J$, see Figure~\ref{Fig17}.
\begin{figure}[ht]
	\begin{center}
		\begin{minipage}{6cm}
			\begin{center} 
				\begin{overpic}[height=5cm]{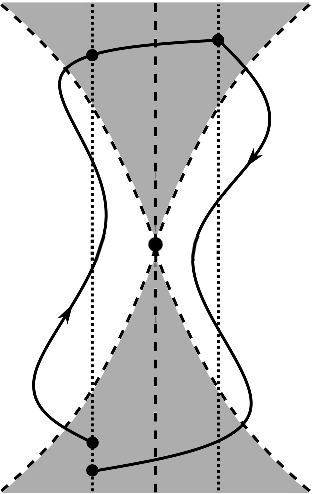} 
					\put(38.25,101.5){{\tiny$x=\mu$}}
					\put(12.5,-3){{\tiny$x=-\mu$}}
					\put(39.5,87){$p$}
					\put(9.5,3.5){$p_R$}
					\put(21.5,9){$p_L$}
					\put(21,85){$p_*$}
				\end{overpic}
				
				$\;$
				
				$(a)$ $\delta=-1$.
			\end{center}
		\end{minipage}
		\begin{minipage}{6cm}
			\begin{center} 
				\begin{overpic}[height=5cm]{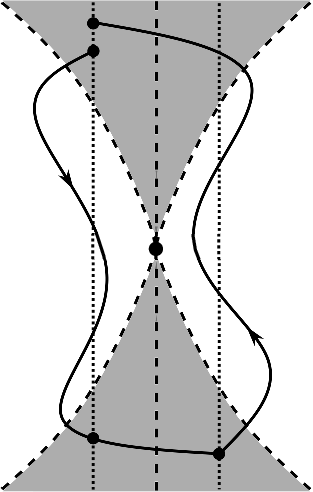} 
					\put(38.25,-3){{\tiny$x=\mu$}}
					\put(12.5,101.5){{\tiny$x=-\mu$}}
					\put(47,5){$p$}
					\put(9.25,94.5){$p_R$}
					\put(21.5,88){$p_L$}
					\put(21,13){$p_*$}
				\end{overpic}
				
				$\;$
				
				$(b)$ $\delta=1$.
			\end{center}
		\end{minipage}
	\end{center}
\caption{Phase portraits of the hysteretic system~\eqref{3} with $\mu>0$ illustrating the definitions of $P_h^L$ and $P_h^R$. In this case $k_L=k_R$ so the sets $A_h^L$ and $A_h^R$ (shaded) are identical. Consider the intersection between the shaded area and the line $x=\mu$. The section $S_h^{\mu,\varepsilon}$ is the connected component of this intersection containing $p$.}\label{Fig17}
\end{figure}
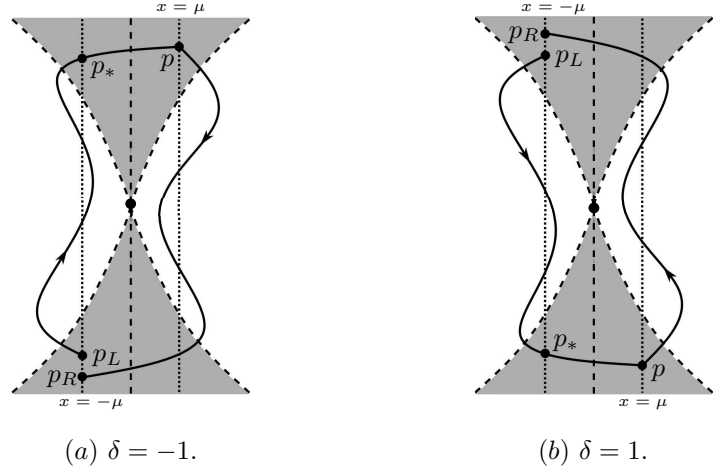

In the simplest case of quadratic tangencies, i.e.~$k_J=1$,
this type of restriction was employed by Makarenkov~\cite{Mak2017}, who obtained the formulas
\begin{align*}
	T^J_h(\mu,y)
	&=-\frac{2}{b_{00}^J}y+\frac{2}{a_{01}^J}\frac{\mu}{y}+R^J(\mu,y),
	\vspace{0.2cm} \\
	P^J_h(\mu,y) &=-y + \alpha_2^J y^2 + \frac{2 b_{00}^J}{a_{01}^J}\frac{\mu}{y}+S^J(\mu,y),
\end{align*}
where $R^J$ and $S^J$ contain higher order terms. Notice the $\mu/y$-terms are small because with $(\mu,y) \in W_h^J$, where $k_J=1$, we have $|\mu/y|<|y|^\frac{3}{2}$.

In view of the above remarks, we define $P_h^J$ and the corresponding evolution time function $T_h^J$ as follows. Given $(\mu,y)\in W_h^R$, let $P_h^R(\mu,y)$ be the second coordinate of the first point $p_R$ at which the forward orbit of $p=(\mu,y)$ under $X^R$ intersects $x=-\mu$, and let $T_h^R(\mu,y)$ be the corresponding evolution time, see Figure~\ref{Fig17}. Given $(\mu,y)\in W_h^L$, let $P_h^L(\mu,y)$ be the second coordinate of the
\emph{second} point $p_L$ at which the backward orbit of $p=(\mu,y)$ under $X^L$ intersects $x=-\mu$ in $A_h^L$, and let $T_h^L(p)$ be the corresponding evolution time.

Notice that the for left map $P_h^L$, it is no longer helpful to evolve backwards and take the first intersection $p_*$ with $x=-\mu$, 
because our goal is to identify the limit cycle. As seen, for example in Figure~\ref{Fig14}$(b)$, the limit cycle corresponds to the second intersection with $x=-\mu$. Recall, $\eta^J_h$ is given by~\eqref{57}.

\begin{lemma}\label{L1}
	Consider the hysteretic system~\eqref{3} with a $(2k_L,2k_R)$-monodromic tangential singularity at the origin.
	There exists $\varepsilon>0$ such that its evolution time functions $T^J_h$ are analytic in $W^J_h$ and given by 
	\begin{equation}\label{HalfPeriod}
		T^J_h(\mu,y)=T_0^J(y)+\frac{\mu}{y^{2k_J-1}}\sum_{i,j\geqslant0}\tau_{ij}^{J,h}\left(\frac{\mu}{y^{2k_J}}\right)^iy^j,
	\end{equation}
	with $\tau_{00}^{J,h}=2/a_{0,2k_J-1}^J$. Moreover, its half-return maps $P^J_h$ are also analytic in $W^J_h$ and given by
	\begin{equation}\label{HalfPoincare}
		P^J_h(\mu,y)=P_0^J(y) + \frac{\mu}{y^{2k_J-1}}\left(\eta^J_h+R_1^J(y)\right)+R_2^J(\mu,y),
	\end{equation}
	where $R_1^J\in O(y)$ and $R_2^J\in O(\mu^2)$.
\end{lemma}

The proof of Lemma~\ref{L1} is postponed to Appendix~\ref{AppA}.
To study the limit cycles of the hysteretic system~\eqref{3}, we now combine the two half-return maps, working on the set 
\begin{equation}
		W_h=\left\{(\mu,y)\in(0,\varepsilon)\times(-\varepsilon,\varepsilon)\colon \mu<|y|^{2k_M+\frac{1}{2}}, \, \delta y<0\right\},
\end{equation}
and notice $W_h=W_h^L\cap W_h^R$ by~\eqref{48W}. Specifically, we define
\[
	T_h(\mu,y)=T^R_h(\mu,y)-T^L_h(\mu,y), \quad D_h(\mu,y)=\delta\big(P^R_h(\mu,y)-P^L_h(\mu,y)\big),
\]
where $\delta=\pm1$ is given by~\eqref{eq:delta}. Notice that $D_h(0,y)=D_0(y)$,
where $D_0(y)=\delta\left(P^R_0(y)-P^L_0(y)\right)$ as in~\eqref{13}. Let
\[
	\vartheta_h=\left\{\begin{array}{ll}
				-2/a_{0,2k_L-1}^L, &\text{if } k_L>k_R, 
				\vspace{0.1cm} \\
				2/a_{0,2k_R-1}^R-2/a_{0,2k_L-1}^L, &\text{if } k_L=k_R, 
				\vspace{0.1cm} \\
				2/a_{0,2k_R-1}^R, &\text{if } k_L<k_R.
	\end{array}\right.
\] 
and recall $\kappa_h$ is defined in~\eqref{57}. The following result follows directly from Lemma~\ref{L1}.

\begin{corollary}\label{C1}
		Consider the hysteretic system~\eqref{3} with a $(2k_L,2k_R)$-monodromic tangential singularity at the origin. There exists $\varepsilon>0$ such that its evolution time $T_h$ and displacement~$D_h$ functions are analytic in $W_h$ and given by
	\begin{align}
		T_h(\mu,y) &= T_0(y)+\frac{\mu}{y^{2k_M-1}}\big(\vartheta_h+R_1(y)\big)+R_2(\mu,y),\label{40}
		\vspace{0.2cm} \\
		D_h(\mu,y) &= D_0(y)+\delta\frac{\mu}{y^{2k_M-1}}\big(\kappa_h+S_1(y)\big)+S_2(\mu,y),
	\end{align}
	with $R_1$, $S_1\in O(y)$, and $R_2$, $S_2\in O(\mu^2)$.
\end{corollary}

\subsection{Time-delay}\label{Sec3.2}

We now provide similar extensions of $T^J_0$ and $P^J_0$ to the time-delayed system~\eqref{3x}. Let
\begin{equation}\label{39}
	W^J_{td}=\left\{(\mu,y)\in(0,\varepsilon)\times(-\varepsilon,\varepsilon)\colon \mu<|y|^{a_J+\frac{1}{2}}\right\},
\end{equation}
with $a_J=1+\max\{0,2(k_J-k_K)\}$, and $K\in\{L,R\}$ such that $K\neq J$. We first define $T^R_{td}$ and $P^R_{td}$ for forward evolution in $\Sigma_0^R$.

Given $p=(0,y)$ with $\delta y<0$ and $(\mu,y) \in W^R_{td}$, we follow the forward orbit of $p$ under $X^L$ by a time $t=\mu$, arriving at a point $q_R$. Then we follow the forward orbit of $q_R$ under $X^R$ until it intersects the line $x=0$ at a point $p_R$, see Figure~\ref{Fig12}. 
\begin{figure}[ht]
	\begin{center}
		\begin{minipage}{6cm}
			\begin{center} 
				\begin{overpic}[height=5cm]{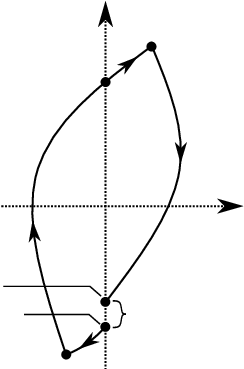} 
					\put(63,39){$x$}
					\put(31,98){$y$}
					\put(30.5,75){$p$}
					\put(-8,22){$p_R$}
					\put(-1.5,14){$p_L$}
					\put(35,12.5){$D_{td}(p)<0$}
					\put(43,88){$q_R$}
					\put(10,1){$q_L$}
				\end{overpic}
				
				$(a)$ $\delta=-1$.
			\end{center}
		\end{minipage}
		\begin{minipage}{6cm}
			\begin{center} 
				\begin{overpic}[height=5cm]{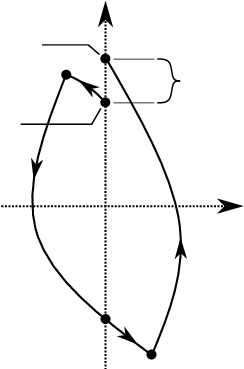} 
					\put(63,39){$x$}
					\put(31,98){$y$}
					\put(31,14){$p$}
					\put(2,87){$p_R$}
					\put(-3,65.5){$p_L$}
					\put(50,76){$D_{td}(p)>0$}
					\put(9,78){$q_L$}
					\put(43,1){$q_R$}
				\end{overpic}
				
				$(b)$ $\delta=1$.
			\end{center}
		\end{minipage}
	\end{center}
\caption{An illustration of the time-delayed framework.}\label{Fig12}
\end{figure}
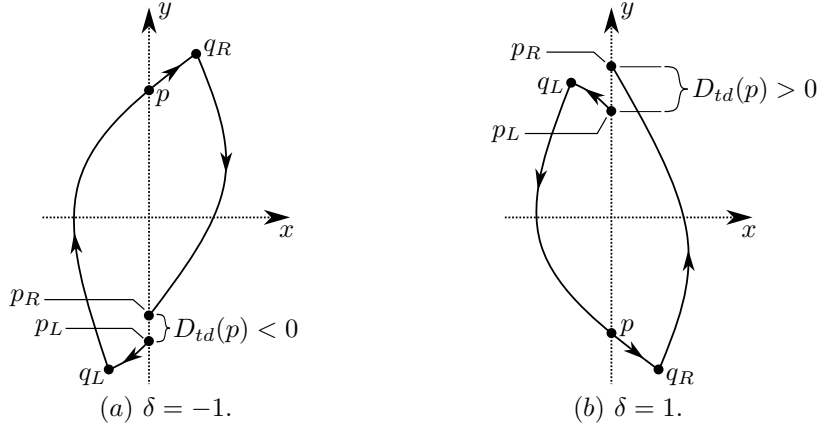
We define $P^R_{td}(\mu,y)$ as the second coordinate of $p_R$, and let $T^R_{td}(\mu,y)$ be the evolution time from $q_R$ to $p_R$.

\begin{lemma}\label{L2}
	Consider the time-delayed system~\eqref{3x} with a $(2k_L,2k_R)$-monodromic tangential singularity at the origin. There exists $\varepsilon>0$ such that the evolution time function $T^R_{td}$ is analytic in $W^R_{td}$ and given by 
	\begin{equation}\label{TDHalfPeriod}
		T^R_{td}(\mu,y)=T_0^R(y)+\frac{\mu}{y^{a_R-1}}\sum_{i,j\geqslant0}\tau_{ij}^{R,td}\left(\frac{\mu}{y^{a_R}}\right)^iy^j,
	\end{equation}
	where
	\[\tau_{00}^{R,td}=\left\{\begin{array}{ll}
		a_{0,2k_L-1}^L/a_{0,2k_R-1}^R, & \text{if } k_R>k_L, 
		\vspace{0.1cm} \\
		a_{0,2k_L-1}^L/a_{0,2k_R-1}^R-2b_{00}^L/b_{00}^R, &\text{if } k_R=k_L, 
		\vspace{0.1cm} \\
		-2b_{00}^L/b_{00}^R, &\text{if } k_R<k_L.
	\end{array}\right.\]
	Moreover, its half-return map $P^R_{td}$ is also analytic in $W^R_{td}$ and given by
	\begin{equation}\label{TDHalfPoincare}
		P^R_{td}(\mu,y)=P_0^R(y)+\frac{\mu}{y^{a_R-1}}\left(\eta^R_{td}+R_1^R(y)\right)+R_2^R(\mu,y),
	\end{equation}
	with $\eta^R_{td}$ given by~\eqref{37},	$R_1^R\in O(y)$, and $R_2^R\in O(\mu^2)$.
\end{lemma}

We postpone the proof of Lemma~\ref{L2} to Appendix~\ref{AppA}. We now define $T^L_{td}$ and $P^R_{td}$. As in the hysteretic setting, care is needed to define these as they involve backwards evolution.

Given $p=(0,y)$ with $\delta y<0$ and $(\mu,y)\in W^L_{td}$, we follow the backward orbit of $p$ under $X^L$ until arriving at a point $q_L$ with the following property: if we follow the backward orbit of $q_L$ under $X^R$ by a time $t=-\mu$, then the orbit arrives precisely at the line $x=0$, at a point $p_L$. We define $P^L_{td}(\mu,y)$ as the second coordinate of $p_L$. The function $T^L_{td}(\mu,y)$ is defined as the evolution time from $p$ to $q_L$.

\begin{lemma}\label{L3}
	Consider the time-delayed system~\eqref{3x} with a $(2k_L,2k_R)$-monodromic tangential singularity at the origin. There exists $\varepsilon>0$ such that its evolution time function $T^L_{td}$ is analytic in $W^L_{td}$ and given by 
	\[
		T^L_{td}(\mu,y)=T_0^L(y)+\frac{\mu}{ y^{2(k_L-k_R)}}\sum_{i,j\geqslant0}\tau_{ij}^{L,td}\left(\frac{\mu}{y^{a_L}}\right)^iy^j,
	\]
	where
	\[
		a_L=1+\max\{0,2(k_L-k_R)\}, \quad \tau_{00}^{L,td}=\frac{a_{0,2k_R-1}^R}{a_{0,2k_L-1}^L}.
	\]
	Moreover, its half-return map $P^L_{td}$ is also analytic in $W^L_{td}$ and given by
	\[
		P^L_{td}(\mu,y)=P_0^L(y)+\frac{\mu}{y^{a_L-1}}\left(\eta^L_{td}+R_1^L(y)\right)+R_2^L(\mu,y),
	\]
	with $\eta^L_{td}$ given by~\eqref{37},	$R_1^L\in O(y)$, and $R_2^L\in O(\mu^2)$.
\end{lemma}

The proof of Lemma~\ref{L3} is postponed to Appendix~\ref{AppA}.
Similar to the hysteretic case, we consider the evolution time and displacement functions $T_{td},D_{td}\colon W_{td}\to\mathbb{R}$ given by
\[
	T_{td}(\mu,y)=T^R_{td}(\mu,y)-T^L_{td}(\mu,y), \quad D_{td}(\mu,t)=\delta\big(P^R_{td}(\mu,y)-P^L_{td}(\mu,y)\big),
\]
where $W_{td} = W^R_{td} \cap W^L_{td}$ is defined by
\begin{equation*}
	W_{td}=\left\{(\mu,y)\in(0,\varepsilon)\times(-\varepsilon,\varepsilon)\colon \mu<|y|^{a_M+\frac{1}{2}}, \; \delta y<0\right\},
\end{equation*}
see Figure~\ref{Fig12}. Recall $\kappa_{td}$ is given by~\eqref{36} and let
\[
	\vartheta_{td}=\left\{\begin{array}{ll}
		-\tau_{00}^{L,td}, &\text{if } k_L>k_R, 
		\vspace{0.1cm} \\
		\tau_{00}^{R,td}-\tau_{00}^{L,td}, &\text{if } k_L=k_R, 
		\vspace{0.1cm} \\
		\tau_{00}^{R,td}, &\text{if } k_L<k_R.
	\end{array}\right.
\] 
The following result follows directly from Lemmas~\ref{L2} and~\ref{L3}.

\begin{corollary}\label{C2}
	Consider the time-delayed system~\eqref{3x} with a $(2k_L,2k_R)$-monodromic tangential singularity at the origin. There exists $\varepsilon>0$ such that its evolution time $T_{td}$ and displacement $D_{td}$ functions are analytic in $W_{td}$ and given by
	\begin{align}
		T_{td}(\mu,y) &= T_0(y)+\frac{\mu}{y^{2|k_R-k_L|}}\big(\vartheta_{td}+R_1(y)\big)+R_2(\mu,y),\label{40x}
		\vspace{0.2cm} \\
		D_{td}(\mu,y) &= D_0(y)+\delta\frac{\mu}{y^{2|k_R-k_L|}}\big(\kappa_{td}+S_1(y)\big)+S_2(\mu,y),
	\end{align}
	with $R_1$, $S_1\in O(y)$, and $R_2$, $S_2\in O(\mu^2)$.
\end{corollary}

\section{Proofs of the main results}\label{Sec4}

\begin{proof}[Proof of Theorem~\ref{HystereticHLB}]
It follows from Corollary~\ref{C1} and~\eqref{13} that
\begin{equation}\label{42}
	D_h(\mu,y)=V_{2n}y^{2n}+O(y^{2n+1})+\delta\kappa_h\frac{\mu}{y^{2k_M-1}}+\frac{\mu}{y^{2k_M-1}}O(y)+O(\mu^2).
\end{equation}
Consider the blow-up-variables $(\nu,z)$ (see~\cite{AlvFerJar2011}) characterized by
\begin{equation}\label{43}
	\mu=\nu^N, \quad y=\nu z,
\end{equation}
with $N=2(n+k_M)-1$, and notice that it is a well-defined change of variables if $\nu\neq0$. Consider 
\begin{equation}\label{44}
	\mathcal{D}_h(\nu,z)=\frac{z^{2k_M-1}}{\nu^{2n}}D_h(\nu^N,\nu z),
\end{equation}
and notice from~\eqref{42} that $\mathcal{D}$ is analytic in $z\neq0$ and given by,
\[
	\mathcal{D}_h(\nu,z)=V_{2n}z^N+O(\nu z^{N+1})+\delta\kappa_h+O(\nu z)+O(\nu^{2(N-n)}).
\]
Notice now that
\[
	z_0:=\left(\frac{-\delta\kappa_h}{V_{2n}}\right)^{1/N}=\delta\operatorname{sgn}(V_{2n})\left|\frac{\kappa_h}{V_{2n}}\right|^{1/N}\neq0
\]
is the unique real solution of $\mathcal{D}_h(0,z_0)=0$, where we recall $\kappa_h<0$. In particular, $-\delta z_0>0$ if, and only if, $\operatorname{sgn}V_{2n}=-1$. Since 
\[
	\frac{\partial\mathcal{D}_h}{\partial z}(0,z_0)=NV_{2n}z_0^{N-1}\neq0,
\]
it follows from the Implicit Function Theorem that there is a unique analytic function 
\[
	Z(\nu)=\sum_{i\geqslant0}z_i\nu^i,
\]
such that $\mathcal{D}_h\big(\nu,Z(\nu)\big)=0$ and $Z(0)=z_0$. From~\eqref{44} now follows
\[
	D_h\big(\nu^N,\nu Z(\nu)\big)=\nu^{2n}\mathcal{D}_h\big(\nu,Z(\nu)\big)=0.
\]
Therefore, we have from~\eqref{43} that
\[
	\mathcal{P}_h(\mu):=\mu^{1/N}\cdot Z\big(\mu^{1/N}\big)=\sum_{i\geqslant0}z_i\mu^{(1+i)/N}
\]
satisfies $D_h\big(\mu,\mathcal{P}_h(\mu)\big)=0$. Notice that $\mathcal{P}_h(\mu)$ is a well-defined solution of $D_h$ since $N>2k_J+1/2$ for $J\in\{L,M\}$. This proves~\eqref{45}, while~\eqref{46} follows by replacing $y=\mathcal{P}_h(\mu)$ in~\eqref{40} and using~\eqref{49T}.
	
To prove that the periodic orbit associated to $\mathcal{P}_h(\mu)$ is a hyperbolic limit cycle, we notice that $\operatorname{sgn}z_0=-\delta$, $V_{2n}<0$, and $\kappa_h<0$ implies
\begin{equation}\label{55}
	\lim\limits_{\mu\to0^+}\mu^{-\frac{2n-1}{N}}\frac{\partial D_h}{\partial y}\big(\mu,\mathcal{P}_h(\mu)\big)=2nV_{2n}z_0^{2n-1}-\frac{2(2k_M-1)\delta\kappa_h}{z_0^{2k_M}}\neq0.
\end{equation}
We now prove that the limit cycle is attracting. To this end, observe that the sign of~\eqref{55} is equal to $\delta$. Therefore, if for each $\mu>0$ we let $d_\mu\colon(-\varepsilon,\varepsilon)\to\mathbb{R}$ be given by $d_\mu(w)=D_h\big(\mu,\mathcal{P}(\mu)+w\big)$,
then
\begin{equation}\label{56}
	d_\mu(w)=\delta aw+O(w^2)
\end{equation}
is analytic for some $a>0$. Observe now that if $\delta=-1$ (resp. $\delta=1$), then $w>0$ is located in the unbounded (resp. bounded) region limited by the limit cycle, while $w<0$ is located in the bounded (resp. unbounded) region, see Figure~\ref{Fig8}.
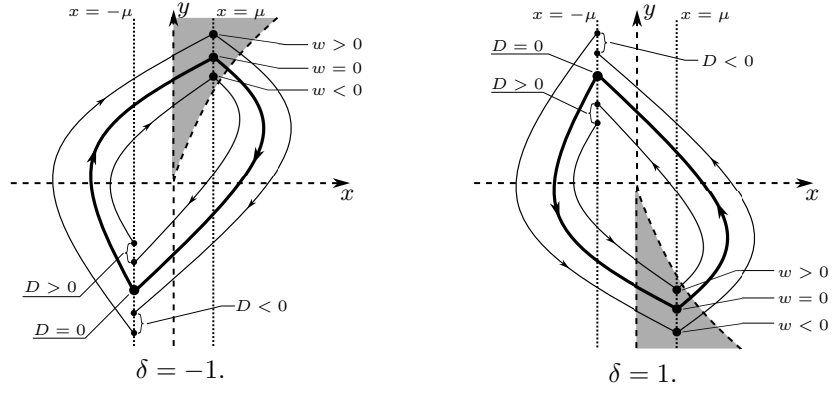
\begin{figure}[ht]
	\begin{center}
		\begin{minipage}{6cm}
			\begin{center} 
				\begin{overpic}[width=4.5cm]{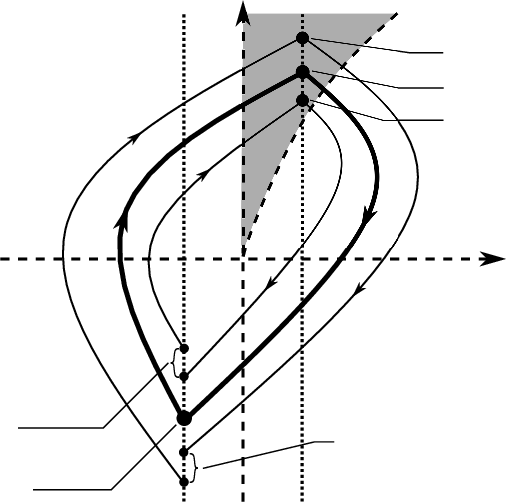} 
					\put(88.5,87.5){\tiny$\tiny w>0$}
					\put(88.5,80.5){\tiny$w=0$}
					\put(88.5,74){\tiny$w<0$}
					\put(97,43){\small$x$}
					\put(49,98){\small$y$}
					\put(17,97){{\tiny$x=-\mu$}}
					\put(60.5,97){{\tiny$x=\mu$}}
					\put(6.75,3.5){\tiny$D=0$}
					\put(4,16){\tiny$D>0$}
					\put(66,10.25){\tiny$D<0$}
				\end{overpic}
					
				$\delta=-1$.
			\end{center}
		\end{minipage}
		\begin{minipage}{6cm}
			\begin{center} 
				\begin{overpic}[width=4.5cm]{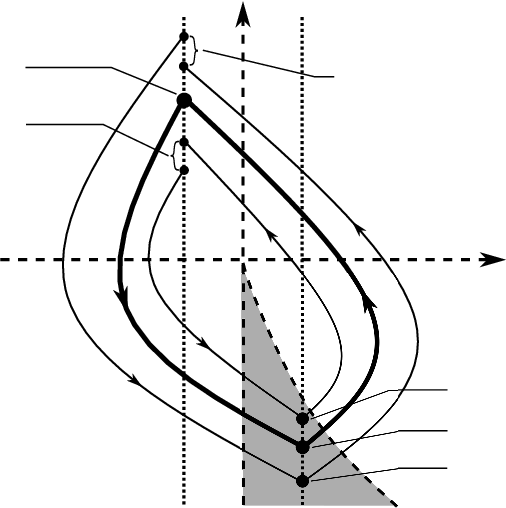} 
					\put(89,21.5){\tiny$\tiny w>0$}
					\put(89,13.75){\tiny$w=0$}
					\put(89,5.75){\tiny$w<0$}
					\put(97,44){\small$x$}
					\put(49.5,98){\small$y$}
					\put(17,97){{\tiny$x=-\mu$}}
					\put(60.5,97){{\tiny$x=\mu$}}
					\put(5,87.5){\tiny$D=0$}
					\put(5,76.5){\tiny$D>0$}
					\put(66.5,83){\tiny$D<0$}
				\end{overpic}
					
				$\delta=1$.
			\end{center}
		\end{minipage}
	\end{center}
\caption{An illustration of the dynamics around the limit cycle. The shaded area is the region $A^R_h \cap A^L_h$, see~\eqref{48A}, with $x>0$ and $\delta y<0$.}\label{Fig8}
\end{figure}
This in addition with~\eqref{56} proves that the limit cycle is attracting.
\end{proof}

\begin{proof}[Proof of Theorem~\ref{TimeDelayedHLB}]
The proof follows similarly to that of Theorem~\ref{HystereticHLB}. From Corollary~\ref{C2} and~\eqref{13},
\begin{equation}\label{42x}
	D_{td}(\mu,y)=V_{2n}y^{2n}+O(y^{2n+1})+\delta\kappa_{td}\frac{\mu}{y^{2|k_R-k_L|}}+\frac{\mu}{y^{2|k_R-k_L|}}O(y)+O(\mu^2).
\end{equation}
Consider the blown-up variables $(\nu,z)$ given by
\begin{equation}\label{43x}
	\mu=\nu^M, \quad y=\nu z,
\end{equation}
where $M=2(n+|k_R-k_L|)$, and notice that this is a well-defined change of variables if $\nu\neq0$. Let 
\begin{equation}\label{44x}
	\mathcal{D}_{td}(\nu,z)=\frac{z^{2|k_R-k_L|}}{\nu^{2n}}D_{td}(\nu^M,\nu z),
\end{equation}
and notice from~\eqref{42x} that $\mathcal{D}_{td}$ is analytic in $z\neq0$ and given by,
\[
	\mathcal{D}_{td}(\nu,z)=V_{2n}z^M+O(\nu z^{M+1})+|\kappa_{td}|+O(\nu z)+O(\nu^{2(M-n)}),
\]
where we have used $\operatorname{sgn}(\kappa_{td})=\delta$. Observe that $\mathcal{D}_{td}(0,z)=0$ has at most one real solution $z_0$ satisfying $-\delta z_0>0$. Moreover, since $M$ is even, such a solution exists if and only if $V_{2n}<0$, and in this case it is given by
\[
	z_0=-\delta\left|\frac{\kappa_{td}}{V_{2n}}\right|^{1/M}\neq0.
\]
Knowing that
\[
	\frac{\partial\mathcal{D}_{td}}{\partial z}(0,z_0)=MV_{2n}z_0^{M-1}\neq0,
\]
it follows from the Implicit Function Theorem that there is a unique analytic function 
\[
	Z(\nu)=\sum_{i\geqslant0}z_i\nu^i,
\]
such that $\mathcal{D}_{td}\big(\nu,Z(\nu)\big)=0$ and $Z(0)=z_0$. This, and~\eqref{44x}, implies~$D\big(\nu^M,\nu Z(\nu)\big)=0$. From~\eqref{43x} it now follows 
\[
	\mathcal{P}_{td}(\mu):=\mu^{1/M}\cdot Z\big(\mu^{1/M}\big)=\sum_{i\geqslant0}z_i\mu^{(1+i)/M}
\]
satisfies $D_{td}\big(\mu,\mathcal{P}_{td}(\mu)\big)=0$. We remark that $\mathcal{P}_{td}(\mu)$ is a well-defined solution of $D_{td}$ since $M>3/2+2|k_R-k_L|$. This proves~\eqref{45x}, while~\eqref{46x} now follows by replacing $y=\mathcal{P}_{td}(\mu)$ in~\eqref{40x} and using~\eqref{49T}.
	
To prove that the periodic orbit associated to $\mathcal{P}_{td}(\mu)$ is a hyperbolic limit cycle, we notice that $\operatorname{sgn}z_0=-\delta$, $V_{2n}<0$, and $\operatorname{sgn}\kappa_{td}=\delta$ implies
\begin{equation}\label{55x}
	\lim\limits_{\mu\to0^+}\mu^{-\frac{2n-1}{M}}\frac{\partial D_{td}}{\partial y}\big(\mu,\mathcal{P}_{td}(\mu)\big)=2nV_{2n}z_0^{2n-1}-2|k_R-k_L|\frac{|\kappa_{td}|}{z_0^{2|k_R-k_L|+1}}\neq0.
\end{equation}
The limit cycle is attracting because the sign of~\eqref{55x} is equal to~$\delta$ so nearby orbits behave analogously to those in Figure~\ref{Fig8} in the hysteretic case.

It follows from the definitions of the half-return maps and displacement function provided in Section~\ref{Sec3.2} that the periodic orbits studied in this proof intersect $\Sigma$ in exactly two points $p$,~$q$, and that between them the system switches from $X^L$ to $X^R$ or vice-versa exactly once. We now prove that these conditions are not restrictive in the sense that any local periodic orbit satisfies these conditions

To this end, recall that $\Phi^J(t,x,y)$ denotes the solution of $X^J$ with $\Phi^J(0,x,y)=(x,y)$ and consider the sets
\[\begin{split}
	\eta_L=\{\Phi^L(\mu,0,-\delta y)\colon 0\leqslant y\leqslant\varepsilon\}, &\quad \eta_R=\{\Phi^R(\mu,0,\delta y)\colon 0\leqslant y\leqslant\varepsilon\},
	\vspace{0.2cm} \\
	\Sigma^L=\{(x,y)\in U\colon x\leqslant 0\}, &\quad \Sigma^R=\{(x,y)\in U\colon x\geqslant 0\},	
\end{split}\]
see Figure~\ref{Fig11}.
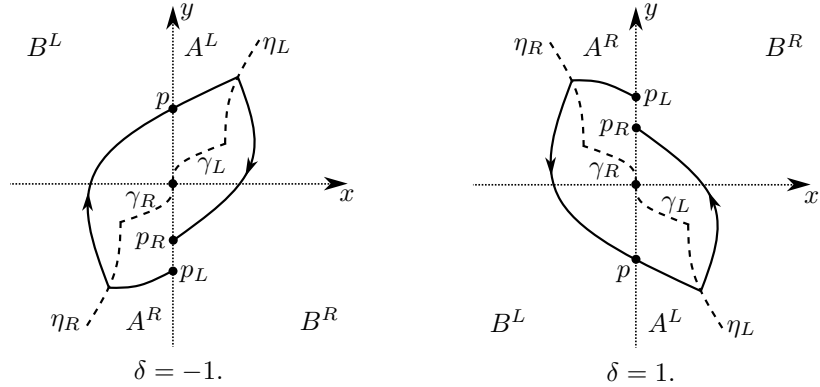
\begin{figure}[ht]
	\begin{center}
		\begin{minipage}{6cm}
			\begin{center} 
				\begin{overpic}[width=4.5cm]{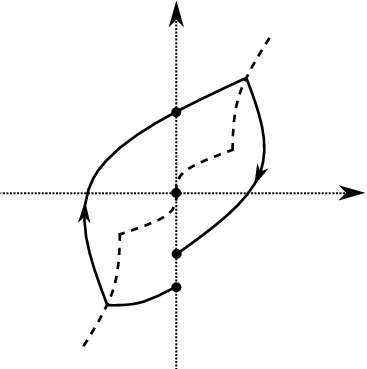} 
					\put(96.5,42.5){$x$}
					\put(50,97){$y$}
					\put(55,52){$\gamma_L$}
					\put(74,87){$\eta_L$}
					\put(34,42){$\gamma_R$}
					\put(12,6){$\eta_R$}
					\put(43,71){$p$}
					\put(37,30){$p_R$}
					\put(50,20.5){$p_L$}
					\put(51,85){$A^L$}
					\put(85,5){$B^R$}
					\put(34,5){$A^R$}
					\put(5,85){$B^L$}
				\end{overpic}
				
				$\delta=-1$.
			\end{center}
		\end{minipage}
		\begin{minipage}{6cm}
			\begin{center} 
				\begin{overpic}[width=4.5cm]{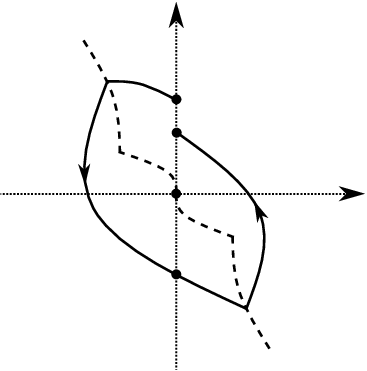} 
					\put(96.5,42.5){$x$}
					\put(50,97){$y$}
					\put(55,41){$\gamma_L$}
					\put(74,5){$\eta_L$}
					\put(34,51){$\gamma_R$}
					\put(12,87){$\eta_R$}
					\put(42,21){$p$}
					\put(37,63.25){$p_R$}
					\put(50,72){$p_L$}
					\put(51,5){$A^L$}
					\put(85,85){$B^R$}
					\put(32,85){$A^R$}
					\put(5,5){$B^L$}
				\end{overpic}
				
				$\delta=1$.
			\end{center}
		\end{minipage}
	\end{center}
\caption{An illustration of the orbits considered in the proof of Theorem~\ref{TimeDelayedHLB}. The second intersection point $q$ occurs when $p_L=p_R$.}\label{Fig11}
\end{figure}
Notice that $\eta_J\cap\Sigma^J=\emptyset$ and thus every orbit passing through $\Sigma$ cannot return unless it switches from $X^L$ to $X^R$ or vice-versa exactly once. Consider now
\[
	\gamma_J=\{\Phi^J(t,0,0)\colon 0\leqslant t\leqslant\mu\}, \quad \Gamma^J=\gamma_J\cup\eta_J,
\]
let $A^R$ (resp. $A^L$) be the region between $\Sigma$ and $\Gamma^R$ (resp. $\Gamma^L$), and let $B^R$ (resp. $B^L$) be the complement of $A^L$ (resp. $A^R$) in $\Sigma^R$ (resp. $\Sigma^L$). Notice that the dynamics on $A^J\cup B^J$ is governed by $X^J$ and thus we have uniqueness of solutions on these regions. In particular, any periodic orbit will intersect $\Sigma$ at exactly two points $p$ and $q$.
\end{proof}

\appendix

\section{Proofs of the technical lemmas}\label{AppA}

\begin{proof}[Proof of Lemma~\ref{L1}]
	
The proof is divided into five steps. In Step 1 we use the series expansion~\eqref{2} of $X^J$ to express the solution to $X^J$ in a neighborhood of the origin. In Step 2 we use this solution to derive expressions for $T_0^J$ and $P_0^J$, defined by~\eqref{50} and~\eqref{51}. In Step 3 we use the blow-up technique to extend these functions to $T^J_h$ and $P^J_h$, for hysteresis. In Step 4 we compute a bound that shows $T^J_h$ and $P^J_h$ are analytic in $A^J_h$. Finally, in Step 5 we prove that $T^J_h$ and $P^J_h$ indeed correspond to the maps defined geometrically in Section~\ref{Sec3.1}.

\bigskip

\noindent {\bf Step 1: Asymptotic expansion of the solution to \boldmath{$X^J$}.}
For each $J\in\{L,R\}$, let 
\[
	\Phi^J(t,x,y)=(\varphi^J(t,x,y),\psi^J(t,x,y))
\]
denote the solution to $X^J$, given by~\eqref{1} and~\eqref{2}, with initial condition $\Phi^J(0,x,y)=(x,y)$. This solution is analytic because $X^J$ is analytic~\cite{Per2001}*{Section~$2.3$}, and thus we can consider its expansion around $(0,0,0)$. Knowing that $\Phi^J$ satisfies
\begin{equation}\label{4}
	\frac{\partial\Phi^J}{\partial t}(t,x,y)=X^J\circ\Phi^J(t,x,y),
\end{equation}
it follows 
\begin{equation}\label{5}
	\frac{\partial^2\Phi^J}{\partial t^2}(t,x,y)=DX^J\big(\Phi^J(t,x,y)\big)\cdot\big(X^J\circ\Phi^J(t,x,y)\big),
\end{equation}
where $DX^J$ denotes the Jacobian matrix of $X^J$ in relation to $(x,y)$. By substituting $\Phi^J(0,x,y)=(x,y)$ in~\eqref{4} and~\eqref{5} we conclude that 
\begin{equation}\label{6}
	\Phi^J(t,x,y)=(x,y)+X^J(x,y)t+DX^J(x,y)\cdot X(x,y)\frac{t^2}{2}+O(t^3),
\end{equation}
and this extends to
\begin{equation}\label{7}
	\Phi^J(t,x,y)=\sum_{i\geqslant0}\frac{F_i^J(x,y)}{i!}t^i,
\end{equation}
with
\begin{equation}\label{47}
	F_i^J(x,y)=DF_{i-1}^J(x,y)\cdot X^J(x,y), \quad i\geqslant1, \quad \textnormal{and} \quad F_0^J(x,y)=(x,y).
\end{equation}
Therefore, we can use~\eqref{47} to recursively calculate $F^J_i$, and then use~\eqref{7} to obtain the expansion of $\Phi^J$. For example, it follows from~\eqref{47} that $F_1^J=X^J=(f^J,g^J)$, and
\begin{equation}\label{9}
	F_2^J=DX^J\cdot X^J=\left(\begin{array}{cc}
										f^J_x & f^J_y
										\vspace{0.1cm} \\
										g^J_x & g^J_y
								\end{array}\right)
								\left(\begin{array}{c}
									f^J 
									\vspace{0.1cm} \\
									g^J
								\end{array}\right)=
								\left(\begin{array}{c}
									f^J_xf+f^J_yg 
									\vspace{0.1cm} \\
									g^J_xf+g^J_yg 
								\end{array}\right),							
\end{equation}
with the subscripts denoting partial derivatives.
By replacing $F^J_1=(f^J,g^J)$ and~\eqref{9} in~\eqref{6}, and taking only the first component, we obtain
\begin{equation}\label{38}
	\begin{split}
		\displaystyle \varphi^J(t,x,y) &\displaystyle= x+f^J(x,y)t 		
		\vspace{0.2cm} \\		
		&\displaystyle+ \Big(f^J_x(x,y)f^J(x,y)+f^J_y(x,y)g^J(x,y)\Big)\frac{t^2}{2}+O(t^3)
		\vspace{0.2cm} \\			
		&\displaystyle =x+a_{00}^Jt+a_{10}^Jtx+a_{01}^Jty+(a_{10}^Ja_{00}^J+a_{01}^Jb_{00}^J)\frac{t^2}{2}
		\vspace{0.2cm} \\		
		&\displaystyle+ t\sum_{i+j+k\geqslant2}\alpha_{ijk}^Jt^ix^jy^k.
	\end{split}
\end{equation}
In the case $k_J=1$, we have $a_{00}^J=0$, so~\eqref{38} reduces to
\[
	\varphi^J(t,x,y)=x+a_{01}^Jty+\frac{a_{01}^Jb_{00}^J}{2}t^2+r^J(t,x,y),
\]
where 
\[
	\displaystyle r^J(t,x,y) = a_{10}^Jtx+t\sum_{i+j+k\geqslant2}\beta_{ijk}^Jt^ix^jy^k.
\]
For general $k_J\geqslant1$, we have $a_{0,i}^J=0$ for all $i\in\{0,\dots,2k_J-2\}$ (Lemma \ref{LC}), and it is a simple exercise to obtain from the above formulas
\begin{equation}\label{10}
	\varphi^J(t,x,y) = x+a_{0,2k_J-1}^J h^J(t,y)+R^J(t,x,y),
\end{equation}
where
\begin{equation}\label{11}
	\begin{split}
		\displaystyle h^J(t,y)  &\displaystyle= ty^{2k_J-1}+\frac{(2k_J-1)}{2!}b_{00}^Jt^2y^{2k_J-2}
		\vspace{0.2cm} \\
		&\displaystyle+ \frac{(2k_J-1)(2k_J-2)}{3!}(b_{00}^J)^2t^3y^{2k_J-3}+\ldots+\frac{(2k_J-1)!}{(2k_J)!}(b_{00}^J)^{2k_J-1}t^{2k_J},
	\end{split}
\end{equation}
and
\begin{equation}\label{12}
	\displaystyle R^J(t,x,y) = tx\sum_{i+j+k=0}^{2k_J-2}\alpha_{ijk}^Jt^ix^jy^k+t\sum_{i+j+k\geqslant2k_J}\beta_{ijk}^Jt^ix^jy^k.
\end{equation}
We remark that throughout the paper we use $\alpha_{ijk}^J$, $\beta_{ijk}^J\in\mathbb{R}$ to denote coefficients of higher order terms that do not need our full attention. The appearing of such coefficients in different equations do not mean that they are the same. For example, the $\alpha_{ijk}^J$ appearing in~\eqref{38} and~\eqref{12} are not necessarily equal.

\bigskip

\noindent {\bf Step 2: Derivation of \boldmath{$T_0^J$} and \boldmath{$P_0^J$}.}
The hysteretic evolution time function $T^J_h$ is characterized by
\[
	\varphi^J\big(T^J_h(\mu,y),\mu,y\big)=-\mu.
\]
Hence, to obtain $T^J_h$ it would be desirable to apply the Implicit Function Theorem in~\eqref{10} for $t$. However, since $X^J$ is not transversal to $y=0$ at the origin, it follows that~\eqref{10} has no linear term in $t$ independent of $\mu$ and $y$. To bypass this we will consider a different set of coordinates, based on blow-up theory~\cite{AlvFerJar2011}, and apply the Implicit Function Theorem on those coordinates. To this end, let
\[
	F^J(t,\mu,y):=\varphi^J(t,\mu,y)+\mu.
\]
We seek to solve $F^J(t,\mu,y)=0$ in $t$. From~\eqref{10}, we have
\begin{equation}\label{15}
	F^J(t,\mu,y)=2\mu+a_{0,2k_J-1}^J h^J(t,y)+R^J(t,\mu,y),
\end{equation}
with $h^J$ and $R^J$ given by~\eqref{11} and~\eqref{12}. Notice that $F^J(t,0,y)=0$ has $t=0$ as a trivial solution. However, it follows from~\cite{NovSil2022} that there there is a second solution $T_0(y)$. For completeness, we now obtain it here. First, consider the blow-up variables $(\zeta,z)$ characterized by $(t,y)=(\zeta z,z)$ and notice that this is a valid change of variables if $y\neq0$. In these new variables we have from~\eqref{10},~\eqref{11} and~\eqref{12} that
\begin{equation}\label{8}
	\varphi^J(\zeta z,0,z) = a_{0,2k_J-1}^Jz^{2k_J}g^J(\zeta)+\zeta z\sum_{i+k\geqslant2k_J}\beta_{i,0,k}^J\zeta^iz^{i+k},
\end{equation}
where 
\begin{equation}\label{17}
	\begin{split}
		\displaystyle g^J(\zeta) &\displaystyle= \zeta+\frac{(2k_J-1)}{2!}b_{00}^J\zeta^2
		\vspace{0.2cm} \\
		&\displaystyle +\frac{(2k_J-1)(2k_J-2)}{3!}(b_{00}^J)^2\zeta^3+\ldots+\frac{(2k_J-1)!}{(2k_J)!}(b_{00}^J)^{2k_J-1}\zeta^{2k_J}.
	\end{split}
\end{equation}
Since $i+k\geqslant 2k_J$ in the summation of~\eqref{8}, it follows that~\eqref{8} has $z^{2k_J}$ as a factor and thus
\begin{equation}\label{16}
	\phi^J(\zeta,z):=\frac{\varphi^J(\zeta z,0,z)}{z^{2k_J}}=a_{0,2k_J-1}^Jg^J(\zeta)+\zeta z\sum_{i+k\geqslant2k_J}\beta_{i,0,k}^J\zeta^iz^{i+j-2k_J},
\end{equation}
is an analytic function. From~\eqref{17} we have
\begin{equation}\label{18}
	g^J(\zeta)=\frac{1}{2k_Jb_{00}^J}\sum_{i=1}^{2k_J}\binom{2k_J}{i}(b_{00}^J\zeta)^i=\frac{1}{2k_Jb_{00}^J}\big((b_{00}^J\zeta+1)^{2k}-1\big),
\end{equation}
with the last equality following from the Binomial Theorem. In particular, it follows from~\eqref{18} that $g^J(\zeta)=0$ if, and only if, $b_{00}^J\zeta+1=\pm1$. This in turn implies $\zeta=0$ or $\zeta=-2/b_{00}^J$. Let~$\zeta_0^J=-2/b_{00}^J$ and notice from~\eqref{16} and~\eqref{18} that
\[
	\phi^J(\zeta_0^J,0)=0, \quad \frac{\partial\phi^J}{\partial\zeta}(\zeta_0^J,0)=-a_{0,2k_J-1}\neq0.
\]
Thus by the Implicit Function Theorem there is a unique analytic function
\[
	\zeta^J(z)=\sum_{i\geqslant0}\zeta_i^Jz^i,
\] 
such that $\phi^J(\zeta^J(z),z)=0$ and $\zeta^J(0)=\zeta_0^J$. In particular, if $z\neq0$ (i.e. $y\neq0$), then it follows from~\eqref{16} that $\varphi^J(y\zeta(y),0,y)=0$. In particular, it follows that
\begin{equation}\label{19}
	T_0^J(y)=y\zeta^J(y)=-\frac{2}{b_{00}^J}y+O(y^2)
\end{equation}
is an analytic solution of
\begin{equation}\label{20}
	F^J(T_0^J(y),0,y)=0.
\end{equation}
We can then conclude from~\eqref{12},~\eqref{15} and~\eqref{20} that
\begin{equation}\label{48}
	\begin{array}{l}
		\displaystyle a_{0,2k_J-1}^Jh^J\big(T_0^J(y),y\big)+R^J\big(T^J(y),0,y\big)
		\vspace{0.2cm} \\
		\displaystyle \qquad\qquad\quad =a_{0,2k_J-1}^Jh^J\big(T_0^J(y),y\big)+T_0^J(y)\sum_{i+k\geqslant2k_J}\beta_{i,0,k}^J\big(T^J_0(y)\big)^iy^k=0.
	\end{array}
\end{equation}

\bigskip

\noindent {\bf Step 3: Derivation of \boldmath{$T^J_h$} and \boldmath{$P^J_h$}.} To extend $T_0^J(y)$ to a solution of $F^J(t,\mu,y)=0$ with $\mu\neq0$, we substitute the form $t=T_0^J(y)+s$ in~\eqref{15} and consider the function
\begin{equation}\label{21}
	H^J(s,\mu,y):=F^J(T_0^J(y)+s,\mu,y).
\end{equation}
We now work on the expansion of $H^J$. To this end, first observe from~\eqref{11} that
\begin{equation}\label{22}
	h^J(t,y)=\frac{1}{2k_Jb_{00}^J}\sum_{i=1}^{2k_J}\binom{2k_J}{i}(b_{00}^Jt)^iy^{2k_J-i}.
\end{equation}
From~\eqref{22} and the Binomial Theorem we obtain,
\begin{equation}\label{23}
	\begin{split}
		h^J\big(T_0^J(y)+s,y\big) &\displaystyle= \frac{1}{2k_Jb_{00}^J}\sum_{i=1}^{2k_J}\binom{2k_J}{i}(b_{00}^J)^i\big(T_0^J(y)+s\big)^iy^{2k_J-i}
		\vspace{0.2cm} \\			
		&\displaystyle= 	\frac{1}{2k_Jb_{00}^J}\sum_{i=1}^{2k_J}\binom{2k_J}{i}(b_{00}^J)^iy^{2k_J-i}\sum_{j=0}^{i}\binom{i}{j}s^jT_0^J(y)^{i-j}
		\vspace{0.2cm} \\			
		&\displaystyle= 	\frac{1}{2k_Jb_{00}^J}\sum_{i=1}^{2k_J}\binom{2k_J}{i}(b_{00}^J)^iy^{2k_J-i}\Bigg(T_0^J(y)^i \Bigg.
		\vspace{0.2cm} \\			
		&\displaystyle \qquad\qquad\qquad\quad \Bigg.+i\cdot s\cdot T_0^J(y)^{i-1}+\sum_{j=2}^{i}\binom{i}{j}s^jT_0^J(y)^{i-j}\Bigg)
		\vspace{0.2cm} \\			
		&\displaystyle= 	\frac{1}{2k_Jb_{00}^J}\sum_{i=1}^{2k_J}\binom{2k_J}{i}\big(b_{00}^JT_0^J(y)\big)^iy^{2k_J-i}
		\vspace{0.2cm} \\			
		&\displaystyle+ 	\frac{s}{2k_J}\sum_{i=1}^{2k_J}\binom{2k_J}{i}i\big(b_{00}^JT_0^J(y)\big)^{i-1}y^{2k_J-i}
		\vspace{0.2cm} \\			
		&\displaystyle+ \frac{1}{2k_Jb_{00}^J}\sum_{i=2}^{2k_J}\binom{2k_J}{i}(b_{00}^J)^iy^{2k_J-i}\sum_{j=2}^{i}\binom{i}{j}s^jT_0^J(y)^{i-j}.
	\end{split}
\end{equation}
Notice from~\eqref{22} that
\begin{equation}\label{24}
	\frac{1}{2k_Jb_{00}^J}\sum_{i=1}^{2k_J}\binom{2k_J}{i}\big(b_{00}^JT_0^J(y)\big)^iy^{2k_J-i}=h^J\big(T_0^J(y),y\big).
\end{equation}
Knowing $\binom{2k_J}{i}i=2k_J\binom{2k_J-1}{i-1}$ and letting $l=i-1$, we obtain from the Binomial Theorem,
\begin{equation}\label{25}
	\begin{array}{l}
		\displaystyle \frac{s}{2k_J}\sum_{i=1}^{2k_J}\binom{2k_J}{i}i\big(b_{00}^JT_0^J(y)\big)^{i-1}y^{2k_J-i}
		\vspace{0.2cm} \\
			
		\displaystyle\qquad =s\sum_{l=0}^{2k_J-1}\binom{2k_J-i}{l}\big(b_{00}^JT_0^J(y)\big)^ly^{2k_J-1-l}=s\cdot\big(b_{00}^JT_0^J(y)+y\big)^{2k_J-1}.
	\end{array}
\end{equation}
Replacing~\eqref{19} at~\eqref{25} we have,
\begin{equation}\label{26}
	\frac{s}{2k_J}\sum_{i=1}^{2k_J}\binom{2k_J}{i}i\big(b_{00}^JT_0^J(y)\big)^{i-1}y^{2k_J-i}=-sy^{2k_J-1}+s\cdot O(y^{2k_J}),
\end{equation}
and we let
\begin{equation}\label{27}
	R_1^J(s,y):=\frac{1}{2k_Jb_{00}^J}\sum_{i=2}^{2k_J}\binom{2k_J}{i}(b_{00}^J)^iy^{2k_J-i}\sum_{j=2}^{i}\binom{i}{j}s^jT_0^J(y)^{i-j}.
\end{equation}
By then substituting \eqref{24},~\eqref{26} and~\eqref{27} into~\eqref{23} we conclude,
\begin{equation}\label{28}
	h^J\big(T_0^J(y)+s,y\big)=h^J\big(T_0^J(y),y\big)-sy^{2k_J-1}+s\cdot O(y^{2k_J})+R_1^J(s,y).
\end{equation}
Next, from~\eqref{15},~\eqref{21} and~\eqref{28} we obtain
\begin{equation}\label{29}
	\begin{split}
		\displaystyle H^J(s,\mu,y) &\displaystyle= 2\mu-a_{0,2k_J-1}^Jsy^{2k_J-1}+s\cdot O(y^{2k_J})
		\vspace{0.2cm} \\
		&\displaystyle +a_{0,2k_J-1}^JR_1^J(s,y)+ a_{0,2k_J-1}^Jh^J\big(T_0^J(y),y\big)+R_2^J(s,\mu,y),
	\end{split}
\end{equation}
with $R_2^J(s,\mu,y):=R^J\big(T_0^J(y)+s,\mu,y)$. It follows from~\eqref{12} that
\begin{equation}\label{30}
	\begin{split}
		\displaystyle R_2^J(s,\mu,y) &\displaystyle= \underbrace{\big(T_0^J(y)+s\big)\mu\sum_{i+j+k=0}^{2k_J-2}\alpha_{ijk}^J\big(T_0^J(y)+s\big)^i\mu^jy^k}_{r_1^J(s,\mu,y)}
		\vspace{0.2cm} \\
		&\displaystyle+ \underbrace{\big(T_0^J(y)+s\big)\sum_{i+j+k\geqslant2k_J}\beta_{ijk}^J\big(T_0^J(y)+s\big)^i\mu^jy^k}_{r_2^J(s,\mu,y)}.
	\end{split}
\end{equation}
Notice that we can divide $r_2^J(s,\mu,y)$ as
\begin{equation}\label{52}
	\begin{split}
		\displaystyle r_2^J(s,\mu,y) &\displaystyle= \underbrace{s\sum_{i+j+k\geqslant2k_J}\beta_{ijk}^J\big(T_0^J(y)+s\big)^i\mu^jy^k}_{r_3^J(s,\mu,y)} + \underbrace{T_0^J(y)\sum_{i+j+k\geqslant2k_J}\beta_{ijk}^J\big(T_0^J(y)+s\big)^i\mu^jy^k}_{r_4^J(s,\mu,y)}.
	\end{split}
\end{equation}
Now, we apply the Binomial Theorem in $r_4^J(s,\mu,y)$ and divide it as
\begin{equation}\label{53}
	\begin{split}
		\displaystyle r_4^J(s,\mu,y) &\displaystyle= T_0^J(y)\sum_{i+j+k\geqslant2k_J}\beta_{ijk}^J\left[\sum_{\ell=0}^{i}\binom{i}{\ell}\big(T_0^J(y)\big)^{i-\ell}s^\ell\right]\mu^jy^k
		\vspace{0.2cm} \\	
		&\displaystyle=  \underbrace{T^J_0(y)\cdot s\sum_{i+j+k\geqslant2k_J}\beta_{ijk}^J\left[\sum_{\ell=1}^{i}\binom{i}{\ell}\big(T_0^J(y)\big)^{i-\ell}s^{\ell-1}\right]\mu^jy^k}_{r_5^J(s,\mu,y)}
		\vspace{0.2cm} \\ 	
		&\displaystyle+ \underbrace{T^J_0(y)\sum_{i+j+k\geqslant2k_J}\beta_{ijk}^J\big(T_0^J(y)\big)^i\mu^jy^k}_{r_6^J(\mu,y)}.
	\end{split}
\end{equation}
Finally, we divide $r_6^J(\mu,y)$ as
\begin{equation}\label{58}
	\begin{split}
		\displaystyle r_6^J(\mu,y) &\displaystyle= \underbrace{T^J_0(y)\sum_{\substack{i+j+k\geqslant2k_J \\ j\geqslant1}}\beta_{ijk}^J\big(T_0^J(y)\big)^i\mu^jy^k}_{r_7^J(\mu,y)} + \underbrace{T^J_0(y)\sum_{i+k\geqslant2k_J}\beta_{i,0,k}^J\big(T_0^J(y)\big)^iy^k}_{r_8^J(\mu,y)}.
	\end{split}	
\end{equation}
Notice that $r_8^J$ is precisely the last summation in~\eqref{48}. Hence, replacing~\eqref{30}, \eqref{52}, \eqref{53} and~\eqref{58} in~\eqref{29}, and using~\eqref{48} to cancel $a_{0,2k_J-1}^Jh^J\big(T_0^J(y),y\big)$ with $r_8^J$, we obtain
\begin{equation}\label{59}
	\begin{split}
		\displaystyle H^J(s,\mu,y) &\displaystyle= 2\mu-a_{0,2k_J-1}^Jsy^{2k_J-1}+s\cdot O(y^{2k_J})+a_{0,2k_J-1}^JR_1^J(s,y)
		\vspace{0.2cm} \\
		&\displaystyle +r_1^J(s,\mu,y)+r_3^J(s,\mu,y)+r_5^J(s,\mu,y)+r_7^J(\mu,y).
	\end{split}
\end{equation}
Consider now the blow-up-variables $(\tau,\nu,z)$ characterized by
\begin{equation}\label{31}
	s=\tau\nu z, \quad \mu=\nu z^{2k_J}, \quad y=z,
\end{equation}
and notice that~\eqref{31} is a well-defined change of variables if $\mu y\neq0$, with inverse 
\begin{equation}\label{32}
	\tau=\frac{sy^{2k_J-1}}{\mu}, \quad \nu=\frac{\mu}{y^{2k_J}}, \quad z=y.
\end{equation}
We now show that the function
\begin{equation}\label{33}
	\mathcal{H}^J(\tau,\nu,z):=\frac{H^J(\tau\nu z,\nu z^{2k_J},z)}{\nu z^{2k_J}},
\end{equation}
is analytic. Indeed, it follows from~\eqref{59} that
\begin{equation}\label{60}
	\begin{split}
		\displaystyle H^J(\tau\nu z,\nu z^{2k_J},z) &\displaystyle= 2\nu z^{2k_J}-a_{0,2k_J-1}^J\tau\nu z^{2k_J}+\tau\nu z O(z^{2k_J})
		\vspace{0.2cm} \\
		&\displaystyle+ a_{0,2k_J-1}^JR_1^J(\tau\nu z,z)+r_1^J(\tau\nu z,\nu z^{2k_J},z)
		\vspace{0.2cm} \\
		&\displaystyle+ r_3^J(\tau\nu z,\nu z^{2k_J},z)+r_5^J(\tau\nu z,\nu z^{2k_J},z)+r_7^J(\nu z^{2k_J},z).
	\end{split}
\end{equation}
From~\eqref{19} and~\eqref{27} we have
\[\begin{split}
	\displaystyle R_1^J(\tau\nu z,z) &\displaystyle= \frac{1}{2k_Jb_{00}^J}\sum_{i=2}^{2k_J}\binom{2k_J}{i}(b_{00}^J)^iz^{2k_J-i}\sum_{j=2}^{i}\binom{i}{j}(\tau\nu z)^j\left(-\frac{2}{b_{00}^J}z+O(z^2)\right)^{i-j}
	\vspace{0.2cm} \\
	&\displaystyle= \frac{1}{2k_Jb_{00}^J}\sum_{i=2}^{2k_J}\binom{2k_J}{i}(b_{00}^J)^iz^{2k_J-i}\sum_{j=2}^{i}\binom{i}{j}(\tau\nu)^jz^jz^{i-j}\left(-\frac{2}{b_{00}^J}+O(z)\right)^{i-j}
	\vspace{0.2cm} \\
	&\displaystyle= \frac{1}{2k_Jb_{00}^J}\sum_{i=2}^{2k_J}\binom{2k_J}{i}(b_{00}^J)^iz^{2k_J-i}(\tau\nu)^2z^i\sum_{j=2}^{i}\binom{i}{j}(\tau\nu)^{j-2}\left(-\frac{2}{b_{00}^J}+O(z)\right)^{i-j}
	\vspace{0.2cm} \\
	&\displaystyle= \frac{\tau^2\nu^2z^{2k_J}}{2k_Jb_{00}^J}\sum_{i=2}^{2k_J}\binom{2k_J}{i}(b_{00}^J)^i\sum_{j=2}^{i}\binom{i}{j}(\tau\nu)^{j-2}\left(-\frac{2}{b_{00}^J}+O(z)\right)^{i-j},
\end{split}\]
so $R_1^J(\tau\nu z,z)\in O(\tau^2 \nu^2 z^{2k_J})$. Similarly, it is not hard to see from~\eqref{30}, \eqref{52}, \eqref{53} and~\eqref{58} that $r_i^J\in O(\nu z^{2k_J+1})$, $i\in\{1,3,5,7\}$. This, in addition with~\eqref{60}, proves that $\mathcal{H}^J$ is analytic and of the form,
\begin{equation}\label{34}
	\mathcal{H}^J(\tau,\nu,z)=2-a_{0,2k_J-1}^J\tau+\tau O(z)+O(\tau^2\nu)+O(z).
\end{equation}
Now, notice that if we let $\tau_{00}^{J,h}=2/a_{0,2k_J-1}^J$, then it follows from~\eqref{34} that
\[
	\mathcal{H}^J(\tau_{00}^{J,h},0,0)=0, \quad \frac{\partial\mathcal{H}^J}{\partial\tau}(\tau_{00}^{J,h},0,0)=-a_{0,2k_J-1}^J\neq0.
\]
Therefore, we have from the Implicit Function Theorem that there is a unique analytic function 
\begin{equation}\label{14}
	\tau^J_h(\nu,z)=\sum_{i,j\geqslant0}\tau_{ij}^{J,h}\nu^iz^j,
\end{equation}
such that $\mathcal{H}^J\big(\tau^J_h(\nu,z),\nu,z\big)=0$, and $\tau^J_h(0,0)=\tau_{00}^{J,h}$. From~\eqref{33} we have
\begin{equation}\label{35}
	H^J\big(\tau^J_h(\nu,z)\nu z,\nu z^{2k_J},z\big)=\nu z^{2k_J}\mathcal{H}^J\big(\tau^J_h(\nu,z),\nu,z\big)=0.
\end{equation}
Back to our original variables, we conclude from~\eqref{32} and~\eqref{35} that
\[
	H^J\left(\tau^J_h\left(\frac{\mu}{y^{2k_J}},y\right)\frac{\mu}{y^{2k_J-1}},\mu,y\right)=0.
\]
From~\eqref{21} follows
\[
	F^J\left(T_0^J(y)+\tau^J_h\left(\frac{\mu}{y^{2k_J}},y\right)\frac{\mu}{y^{2k_J-1}},\mu,y\right)=0,
\]
which in turn implies that $T^J_h$ exists and is given by
\[
	T^J_h(\mu,y)=T_0^J(y)+\tau^J_h\left(\frac{\mu}{y^{2k_J}},y\right)\frac{\mu}{y^{2k_J-1}},
\]
proving~\eqref{HalfPeriod}. Equation~\eqref{HalfPoincare} now follows from~\eqref{HalfPeriod} in addition with
\begin{equation}\label{13x}
	\psi^J(t,x,y)=y+b_{00}^Jt+t\sum_{i+j+k\geqslant1}\alpha_{ijk}^Jt^ix^jy^k,
\end{equation}
which can be obtained similarly to~\eqref{38}. 

\bigskip

\noindent {\bf Step 4: Proof that \boldmath{$T_0^J$} and \boldmath{$P_0^J$} are analytic.} We recall from the Implicit Function Theorem that~\eqref{14} is defined in neighborhood of $(\nu,z)=(0,0)$. Let $\varepsilon_0>0$ small enough such that
\[
	B:=\big\{(\nu,z)\in\mathbb{R}^2\colon \sqrt{|\nu|^2+|z|^2}<\varepsilon_0\big\}
\]
is contained in such a neighborhood. It follows from~\eqref{48A} that if $(x,y)\in A^J_h$, then $y\neq0$ and
\[
	\frac{|x|}{|y|^{2k_J}}<|y|^\frac{1}{2}.
\]
This in addition with~\eqref{32} implies $|\nu|^2<|z|$. Hence, if $(x,y)\in A^J_h$, then
\[
	\sqrt{|\nu|^2+|z|^2}<\sqrt{|z|+|z|^2}<\sqrt{2|z|},
\]
with the last inequality following from $\varepsilon_0>0$ small enough. Thus, if we take $\varepsilon>0$ small enough in~\eqref{48A}, then $(x,y)\in A^J_h$ implies $(\nu,z)\in B$. In particular, $T^J_h$ and $P^J_h$ are well defined and analytic in $A^J_h$.

\bigskip

\noindent {\bf Step 5: \boldmath{$T_0^L$}, \boldmath{$P_0^L$} are associated with the second intersection point.}
Observe from Lemma~\ref{LC} that we can write
\begin{equation}\label{54}
	f^J(x,y)=x\xi^J(x,y)+a_{0,2k_J-1}^Jy^{2k_J-1}+O(y^{2k_J}).
\end{equation}
We have from~\eqref{48A} that if $(x,y)\in A^J_h$, then $x\in o(y^{2k_J})$. Replacing this in~\eqref{54} gives
\[
	f^J|_{A^J_h}(x,y)=a_{0,2k_J-1}^Jy^{2k_J-1}+O(y^{2k_J}).
\]
In particular, if $\varepsilon>0$ is small enough, then it is non-zero for every $(x,y)\in A^J_h$, with its sign depending on the signs of $a_{0,2k_J-1}^J$ and $y$. Therefore, the solution of $X^J$ is transversal to every vertical line $x=\mu$ within $A^J_h$. Such a transversality ensures (as claimed in Section~\ref{Sec3.1}) that if $(\mu,y)\in W^J_h$, then the orbit of $X^J$ through $(\mu,y)$ only has two intersections with $x=+\mu$ and two with $x=-\mu$, within $A^J_h$. Moreover, if $J=R$, then $T^R_h$ is indeed the evolution time of the first intersection point $p_R$ of the forward orbit of $X^R$ through $(\mu,y)$, see Figure~\ref{Fig17}.

To see that $T^L_h$ is associated with the second intersection point $p_L$ of the backward orbit of $X^L$ through $(\mu,y)$, we observe that the same technique used to extend $T_0^J(y)$ to $T^J_h(\mu,y)$, can also be used to extend the trivial solution $t=0$ of~\eqref{15}, obtaining a second evolution time function $T^J_1\colon A^J_h\to\mathbb{R}$. In the case $J=L$, we will see that such a function is associated with the first intersection point $p_*$ with $x=-\mu$.

The extension of $t=0$ to $T^J_1$ is similar to that from $T_0^J$ to $T^J_h$. The main difference is that instead of writing $t=T_0^J(y)+s$ and defining~\eqref{21}, we skip directly to the blow-up coordinates, by considering the function
\[
	\mathcal{F}^J_1(\tau,\nu,z)=\frac{F^J(\tau\nu z,\nu z^{2k_J},z)}{\nu z^{2k_J}}.
\]
Similar to~\eqref{33}, $\mathcal{F}_1$
is analytic and given by
\[
	\mathcal{F}^J_1(\tau,\nu,z)=2+a_{0,2k_J-1}^J\tau+\tau O(\tau\nu z)+O(\tau\nu z)+O(\tau z).
\]
Then using $\tau_{00}^{J,h}=2/a_{0,2k_J-1}$ we obtain
\[
	\mathcal{F}^J_1(-\tau_{00}^{J,h},0,0)=0, \quad \frac{\partial\mathcal{F}^J_1}{\partial\tau}(-\tau_{00}^{J,h},0,0)=a_{0,2k_J-1}^J\neq0,
\]
and by the Implicit Function Theorem,
\begin{equation}\label{61}
	T^J_1(\mu,y)=\frac{\mu}{y^{2k_J-1}}\sum_{i+j=0}^{\infty}\tau_{ij}^{J,h,1}\left(\frac{\mu}{y^{2k_J}}\right)^iy^j,
\end{equation}
with $\tau_{00}^{J,h,1}=-\tau_{00}^{J,h}$. Moreover, the respective
map $P^J_1$ is also analytic on $A^J$ and given by
\[
	P^J_1(\mu,y)=y+\frac{\mu}{y^{2k_J-1}}\left(-\eta^J_h+R_1^{J,1}(y)\right)+R_2^{J,1}(\mu,y),
\]
with $R_1^{J,1}\in O(y)$, $R_2^{J,1}\in O(\mu^2)$, and $\eta^J_h$ given by~\eqref{HalfPoincare}.

Notice from Lemma~\ref{LC} and~\eqref{eq:delta} that $\operatorname{sgn}b_{00}^L=-\delta$, and $\operatorname{sgn}a_{0,2k_L-1}^L=-\delta$. Also, observe that if $(\mu,y)\in W_h$, then $\mu\in o(y^{2k_J})$ and $\operatorname{sgn}y=-\delta$. Hence, it follows from~\eqref{HalfPeriod} and~\eqref{61} that
\[
	T^L_h(\mu,y)<T_1^L(\mu,y)<0,
\]
for every $(\mu,y)\in W_h$. Hence, $T_h^L$ is indeed associated with the second intersection in backward time with $x=-\mu$,
while $T_1^L$ is associated with the first intersection.
\end{proof}

\begin{proof}[Proof of Lemma~\ref{L2}]
	
The proof follows similarly to Lemma~\ref{L1}. We shall prove that $T^R_{td}$ and $P^R_{td}$ are well defined and analytic in
\begin{equation*}
	A^R_{td}:=\left\{(\mu,y)\in(-\varepsilon,\varepsilon)^2\colon |\mu|<|y|^{a_R+\frac{1}{2}}\right\},
\end{equation*}
and in particular in $W^R_{td}$, recalling~\eqref{39}.

Moreover, we shall provide only the third step of the proof, i.e.~the derivation of $T^R_{td}$ and $P^R_{td}$. The first two steps are the same, while the fourth the fifth are similar.

For simplicity, we recall from the proof of Lemma~\ref{L1} that,
\begin{equation}\label{15x}
	h^J(t,y)=\frac{1}{2k_Jb_{00}^J}\sum_{i=1}^{2k_J}\binom{2k_J}{i}(b_{00}^Jt)^iy^{2k_J-i}=\frac{1}{2k_Jb_{00}^J}\big((y+b_{00}^Jt)^{2k_J}-y^{2k_J}\big),
\end{equation}
for $J\in\{L,R\}$. Let
\[
	G^L(\mu,y):=\varphi^L(\mu,0,y), \quad H^L(\mu,y):=\psi^L(\mu,0,y),
\]
and notice from~\eqref{10},~\eqref{12} and~\eqref{13x} that
\begin{equation}\label{18x}
	G^L(\mu,y)=a_{0,2k_L-1}^Lh^L(\mu,y)+\mu R_0^L(\mu,y), \quad H^L(\mu,y)=y+b_{00}^L\mu+\mu S_0^L(\mu,y),
\end{equation}
where
\begin{equation}\label{19x}
	R_0^L(\mu,y)=\sum_{i+k\geqslant2k_L}\beta_{i0k}^L\mu^iy^k, \quad S_0^L(\mu,y)=\sum_{i+k\geqslant1}\alpha_{i0k}^L\mu^iy^k.
\end{equation}
Notice that $T^R_{td}$ is characterized by
\[
	\varphi^R\big(T^R_{td}(\mu,y),G^L(\mu,y),H^L(\mu,y)\big)=0.
\]
Therefore, to obtain $T^R_{td}(\mu,y)$ we consider the function
\begin{equation}\label{17x}
	F^R(t,\mu,y):=\varphi^R\big(t,G^L(\mu,y),H^L(\mu,y)\big),
\end{equation}
and seek to solve $F^R(t,\mu,y)=0$ in $t$. To this end, we now work on its expression. For simplicity, we may omit the dependence of $G^L$ and $H^L$ in relation to $(\mu,y)$. We have from~\eqref{10},~\eqref{18x},~\eqref{19x} and~\eqref{17x} that,
\begin{equation}\label{16x}
	\begin{split}
		\displaystyle F^R(t,\mu,y) &\displaystyle=G^L+a_{0,2k_R-1}^Rh^R(t,H^L)+R^R(t,G^L,H^L)
		\vspace{0.2cm} \\
		\displaystyle &\displaystyle= a_{0,2k_L-1}^Lh^L(\mu,y)+a_{0,2k_R-1}^Rh^R(t,H^L)
		\vspace{0.2cm} \\
		\displaystyle &\displaystyle+ \; \mu R_0^L(\mu,y)+R^R(t,G^L,H^L).
	\end{split}
\end{equation}
Now observe from~\eqref{15x} and~\eqref{18x} that,
\begin{equation}\label{20x}
	\begin{split}
		\displaystyle h^R(t,H^L) &\displaystyle= 	h^R\big(t,y+\mu(b_{00}^L+S_0^L)\big)
		\vspace{0.2cm} \\
		&\displaystyle= 	\frac{1}{2k_Rb_{00}^R}\sum_{i=1}^{2k_R}\binom{2k_R}{i}(b_{00}^Rt)^i\big(y+\mu(b_{00}^L+S_0^L)\big)^{2k_R-i}
		\vspace{0.2cm} \\
		&\displaystyle= 	\frac{1}{2k_Rb_{00}^R}\sum_{i=1}^{2k_R}\Bigg[\binom{2k_R}{i}(b_{00}^Rt)^i\sum_{j=0}^{2k_R-i}\binom{2k_R-i}{j}\mu^j(b_{00}^L+S_0^L)^jy^{2k_R-i-j}\Bigg]
		\vspace{0.2cm} \\
		&\displaystyle=  	\frac{1}{2k_Rb_{00}^R}\sum_{i=1}^{2k_R}\binom{2k_R}{i}(b_{00}^Rt)^i\Bigg(y^{2k_R-i}+
		\vspace{0.2cm} \\
		&\displaystyle 	\qquad\qquad\qquad+(2k_R-i)\mu(b_{00}^L+S_0^L)y^{2k_R-i-1}+
		\vspace{0.2cm} \\
		&\displaystyle \qquad\qquad\qquad\qquad+\sum_{j=2}^{2k_R-i}\binom{2k_R-i}{j}\mu^j (b_{00}^L+S_0^L)^jy^{2k_R-i-j} \Bigg)
		\vspace{0.2cm} \\
		&\displaystyle= 	\frac{1}{2k_Rb_{00}^R}\sum_{i=1}^{2k_R}\binom{2k_R}{i}(b_{00}^Rt)^iy^{2k_R-i}
		\vspace{0.2cm} \\
		&\displaystyle+ 	\frac{\mu(b_{00}^L+S_0^L)}{2k_Rb_{00}^R}\sum_{i=1}^{2k_R}\binom{2k_R}{i}(2k_R-i)(b_{00}^Rt)^iy^{2k_R-1-i}
		\vspace{0.2cm} \\
		&\displaystyle+\frac{1}{2k_Rb_{00}^R}\sum_{i=1}^{2k_R}\Bigg[\binom{2k_R}{i}(b_{00}^Rt)^i \sum_{j=2}^{2k_R-i}\binom{2k_R-i}{j}\mu^j(b_{00}^L+S_0^L)^jy^{2k_R-i-j}\Bigg].
	\end{split}
\end{equation}
Notice from~\eqref{15x} that,
\begin{equation}\label{21x}
	\frac{1}{2k_Rb_{00}^R}\sum_{i=1}^{2k_R}\binom{2k_R}{i}(b_{00}^Rt)^iy^{2k_R-i}=h^R(t,y).
\end{equation}
Knowing that $\binom{2k_R}{i}(2k_R-i)=2k_R\binom{2k_R-1}{i}$, we have from the Binomial Theorem
\begin{equation}\label{22x}
	\begin{array}{l}
		\displaystyle \frac{\mu(b_{00}^L+S_0^L)}{2k_Rb_{00}^R}\sum_{i=1}^{2k_R}\binom{2k_R}{i}(2k_R-i)(b_{00}^Rt)^iy^{2k_R-1-i}
		\vspace{0.2cm} \\
		\displaystyle \qquad\qquad\quad = \frac{\mu(b_{00}^L+S_0^L)}{b_{00}^R}\sum_{i=1}^{2k_R-1}\binom{2k_R-1}{i}(b_{00}^Rt)^iy^{2k_R-1-i}
		\vspace{0.2cm} \\
		\displaystyle \qquad\qquad\qquad\qquad\qquad  = \frac{\mu(b_{00}^L+S_0^L)}{b_{00}^R}\big((y+b_{00}^Rt)^{2k_R-1}-y^{2k_R-1}\big).
	\end{array}
\end{equation}
Also let
\begin{equation}\label{23x}
	\begin{split}
		\displaystyle R_3^R(t,\mu,y) &\displaystyle:= \frac{1}{2k_Rb_{00}^R}\sum_{i=1}^{2k_R}\Bigg[\binom{2k_R}{i}(b_{00}^Rt)^i \sum_{j=2}^{2k_R-i}\binom{2k_R-i}{j}\mu^j(b_{00}^L+S_0^L)^jy^{2k_R-i-j}\Bigg].
	\end{split}
\end{equation}
Then by replacing~\eqref{21x},~\eqref{22x} and~\eqref{23x} in~\eqref{20x} we obtain
\begin{equation}\label{24x}
	\begin{split}
		\displaystyle h^R(t,H^L) &= h^R(t,y)+\frac{b_{00}^L}{b_{00}^R}\big((y+b_{00}^Rt)^{2k_R-1}-y^{2k_R-1}\big)\mu
		\vspace{0.2cm} \\
		&\displaystyle + \frac{b_{00}^L}{b_{00}^R}\big((y+b_{00}^Rt)^{2k_R-1}-y^{2k_R-1}\big)\mu S_0^L(\mu,y)+R_3^R(t,\mu,y),
	\end{split}
\end{equation}
and by replacing~\eqref{24x} in~\eqref{16x} we obtain
\begin{equation}\label{25x}
	\begin{split}
		\displaystyle F^R(t,\mu,y) &\displaystyle= a_{0,2k_L-1}^Lh^L(\mu,y)
		\vspace{0.2cm} \\
		&\displaystyle  + \; a_{0,2k_R-1}^R\frac{b_{00}^L}{b_{00}^R}\big((y+b_{00}^Rt)^{2k_R-1}-y^{2k_R-1}\big)\mu
		\vspace{0.2cm} \\
		&\displaystyle + \; a_{2k_R-1}^Rh^R(t,y)+S_1^R(t,\mu,y)+R_3^R(t,\mu,y)
		\vspace{0.2cm} \\
		&\displaystyle + \; \mu R_0^L(\mu,y)+R^R(t,G^L,H^L),
	\end{split}
\end{equation}
where
\begin{equation}\label{26x}
	S_1^R(t,\mu,y)=\frac{b_{00}^L}{b_{00}^R}\big((y+b_{00}^Rt)^{2k_R-1}-y^{2k_R-1}\big)\mu S_0^L(\mu,y).
\end{equation}
Observe from~\eqref{10},~\eqref{11},~\eqref{12} and~\eqref{13x} that $G^L(0,y)=0$, and $H^L(0,y)=y$. Therefore, it follows from~\eqref{50} and~\eqref{17x} that
\begin{equation}\label{37x}
	F^R(T_0^R(y),0,y)=\varphi^R\big(T_0^R(y),G^L(0,y),H^L(0,y)\big)=\varphi^R\big(T_0^R(y),0,y\big)=0.
\end{equation}
Hence, to solve $F^R(t,\mu,y)=0$ in $t$ we will replace the form $t=T_0^R(y)+s$ and seek for an expression of $s$ as a function of $(\mu,y)$. More precisely, we consider
\begin{equation}\label{41xx}
	H^R(s,\mu,y):=F^R\big(T_0^R(y)+s,\mu,y\big),
\end{equation}
and seek for a solution of $H^R(s,\mu,y)=0$ in $s$. To this end, we also consider a new set of blown-up-variables $(\tau,\nu,z)$ characterized by
\begin{equation*}
	s=\tau \nu z, \quad \mu=\nu z^{a_R}, \quad y=z, \quad a_R=1+\max\{0,2(k_R-k_L)\}.
\end{equation*}
and let 
\begin{equation}\label{30x}
	\mathcal{H}^R(\tau,\nu,z):=\frac{H^R(\tau\nu z,\nu z^{a_R},z)}{\nu z^{2k_R}}.
\end{equation}
We obtain from~\eqref{19},~\eqref{27},~\eqref{28},~\eqref{15x},~\eqref{18x},~\eqref{19x},~\eqref{23x},~\eqref{25x},~\eqref{26x}, and~\eqref{37x} that $\mathcal{H}^R$ is analytic and given by
\begin{equation}\label{33x}
	\mathcal{H}^R(\tau,\nu,z)=C^R-a_{0,2k_R-1}^R\tau+\mathcal{R}^R(\tau,\nu,z),
\end{equation}
with
\begin{equation}\label{40xx}
	C^R=\left\{\begin{array}{ll}
		a_{0,2k_L-1}^L, & \text{if } k_R>k_L, 
		\vspace{0.2cm} \\
		a_{0,2k_L-1}^L-2a_{0,2k_R-1}^Rb_{00}^L/b_{00}^R, &\text{if } k_R=k_L, 
		\vspace{0.2cm} \\
		-2a_{0,2k_R-1}^Rb_{00}^L/b_{00}^R, &\text{if } k_R<k_L, 
	\end{array}\right.
\end{equation}
and
\begin{equation}\label{31x}
	\mathcal{R}^R(\tau,0,0)=0, \quad \frac{\partial\mathcal{R}^R}{\partial\tau}(\tau,0,0)=0.
\end{equation}
Hence, if we let
\[
	\tau_{00}^{R,td}=\left\{\begin{array}{ll}
		a_{0,2k_L-1}^L/a_{0,2k_R-1}^R, & \text{if } k_R>k_L, 
		\vspace{0.2cm} \\
		a_{0,2k_L-1}^L/a_{0,2k_R-1}^R-2b_{00}^L/b_{00}^R, &\text{if } k_R=k_L, 
		\vspace{0.2cm} \\
		-2b_{00}^L/b_{00}^R, &\text{if } k_R<k_L, 
	\end{array}\right.
\]
then we have from~\eqref{33x},~\eqref{40xx} and~\eqref{31x} that,
\[
	\mathcal{H}^R(\tau_{00}^{R,td},0,0)=0, \quad \frac{\partial\mathcal{H}^R}{\partial\tau}(\tau_{00}^{R,td},0,0)=-a_{0,2k_R-1}^R\neq0.
\]
Therefore, it follows from the Implicit Function Theorem that there is a unique analytic function 
\[
	\tau^R_{td}(\nu,z)=\sum_{i,j\geqslant0}\tau_{ij}^{R,td}\nu^iz^j,
\]
such that $\mathcal{H}^R\big(\tau^R_{td}(\nu,z),\nu,v\big)=0$ and $\tau^R_{td}(0,0)=\tau_{00}^{R,td}$. From~\eqref{30x} we have
\begin{equation}\label{36x}
	H^R\big(\tau^R_{td}(\nu,z)\nu z,\nu z^{a_R},z\big)=\nu z^{2k_R}\mathcal{H}^R\big(\tau^R_{td}(\nu,z),\nu,z\big)=0.
\end{equation}
Back to our original variables, we conclude from~\eqref{36x} that
\[
	H^R\left(\tau^R_{td}\left(\frac{\mu}{y^{a_R}},y\right)\frac{\mu}{y^{a_R-1}},\mu,y\right)=0.
\]
This in addition with~\eqref{41xx} implies,
\[
	F^R\left(T_0^R(y)+\tau^R_{td}\left(\frac{\mu}{y^{a_R}},y\right)\frac{\mu}{y^{a_R-1}},\mu,y\right)=0.
\]
Hence, $T^R_{td}(\mu,y)$ exists and is given by
\[
	T^R_{td}(\mu,y)=T_0^R(y)+\tau^R_{td}\left(\frac{\mu}{y^{a_R}},y\right)\frac{\mu}{y^{a_R-1}},
\]
proving~\eqref{TDHalfPeriod}. Equation~\eqref{TDHalfPoincare} now follows by replacing $t=T^R_{td}(\mu,y)$ in
\[
	\psi^R\big(t,G^L(\mu,y),H^L(\mu,y)\big)=y+b_{00}^Rt+b_{00}^L\mu+\sum_{i+j+k\geqslant2}\beta_{ijk}^Rt^i\mu^jy^k,
\]
whose right-hand side follows by replacing~\eqref{18x} and~\eqref{19x} in~\eqref{13x}. That $T^R_{td}$ and $P^R_{td}$ are well defined and analytic in $A^R_{td}$ follows via the same arguments used in the proof of Lemma~\ref{L1}.
\end{proof}

\begin{proof}[Proof of Lemma~\ref{L3}]
The proof follows similarly to that of Lemma~\ref{L2}. We shall prove that $T^J_{td}$ and $P^J_{td}$ are well defined and analytic in $A^J_{td}$, and in particular in $W^J_{td}$. 

The first main difference in this proof is that this time we let
\[
	F^L(t,\mu,y):=\varphi^R\big(-\mu,G^L(t,y),H^L(t,y)\big).
\]
Hence, knowing that $F^L\big(T_0^L(y),0,y)=0$, we let $t=T_0^L(y)+s$ and consider
\[
	H^L(s,\mu,y):=F^L\big(T_0^L(y)+s,\mu,y\big).
\]
We now seek to solve $H^L(s,\mu,y)=0$ for $s$. Similarly to~\eqref{25x}, one can see that
\[
	\begin{split}
		\displaystyle H^L(s,\mu,y) &\displaystyle= a_{0,2k_L-1}^Lh^L\big(T_0^L(y),y\big)-a_{0,2k_L-1}^Lsy^{2k_L-1}+s\cdot O(y^{2k_L})
		\vspace{0.2cm} \\
		&\displaystyle+ \; R_1^L(s,y)+\big(T_0^L(y)+s\big)R_0^L\big(T_0^L(y)+s,y\big)	
		\vspace{0.2cm} \\
		&\displaystyle+ \; \frac{a_{0,2k_R-1}^R}{2k_Rb_{00}^R}\big((y-b_{00}^R\mu)^{2k_R}-y^{2k_R}\big)
		\vspace{0.2cm} \\
		&\displaystyle  + \; \frac{a_{0,2k_R-1}^Rb_{00}^L}{b_{00}^R}\big((y-b_{00}^R\mu)^{2k_R-1}-y^{2k_R-1}\big)T_0^L(y)
		\vspace{0.2cm} \\
		&\displaystyle + \; \frac{a_{0,2k_R-1}^Rb_{00}^L}{b_{00}^R}\big((y-b_{00}^R\mu)^{2k_R-1}-y^{2k_R-1}\big)\Big(s+\big(T_0^L(y)+s\big)S_0^L\big(T_0^L(y)+s,y\big)\Big)
		\vspace{0.2cm} \\
		&\displaystyle + \; a_{0,2k_R-1}^RR_3^R\big(-\mu,T^L_0(y)+s,y\big)+R^R\big(-\mu,G^L,L^L\big).
	\end{split}
\]
The main difference in this proof is a special attention on the expression of $R_3^R$:
\begin{equation}\label{42xx}
	\begin{split}
		\displaystyle R_3^R\big(-\mu,T_0^L(y)+s,y\big) &\displaystyle= \frac{1}{2k_Rb_{00}^R}\sum_{i=1}^{2k_R}\Bigg[\binom{2k_R}{i}(-b_{00}^R\mu)^i
		\vspace{0.2cm} \\ 
		&\displaystyle\times \sum_{j=2}^{2k_R-i}\binom{2k_R-i}{j}\big(T_0^L(y)+s\big)^j(b_{00}^L+S_0^L)^jy^{2k_R-i-j}\Bigg],
	\end{split}
\end{equation}
recall~\eqref{23x}. The special attention is the fact that after the blow-up, the $i=1$ term in~\eqref{42xx} remains outside the remaining function $\mathcal{R}^L$ (recall equations~\eqref{33x} and~\eqref{31x}).
	
More precisely, we consider the new set of blown-up variables $(\tau,\nu,z)$ characterized by
\[
	s=\tau\nu z^{b_L}, \quad \mu=\nu z^{a_L}, \quad y=z,
\]
with $a_L=1+\max\{0,2(k_L-k_R)\}$, $b_L=1+\max\{0,2(k_R-k_L)\}$, and let
\[
	\mathcal{H}^L(\tau,\nu,z):=\frac{H^L(\tau\nu z^{b_L},\nu z^{a_L},z)}{\nu z^{2k_M}},
\]
where $k_M=\max\{k_L,k_R\}$. It now follows similarly to~\eqref{33x} and~\eqref{31x}, that $\mathcal{H}^L$ is analytic and given by
\[
	\mathcal{H}^L(\tau,\nu,z)=a_{0,2k_R-1}^R-a_{0,2k_L-1}^L\tau+\mathcal{R}^L(\tau,\nu,z),
\]
with
\[
	\mathcal{R}^L(\tau,0,0)=0, \quad \frac{\partial\mathcal{R}^L}{\partial\tau}(\tau,0,0)=0.
\]
The proof is then completed via the Implicit Function Theorem. 	
\end{proof}

\section*{Acknowledgments}

DDN was supported by the São Paulo Research Foundation (FAPESP), grants 2024/15612-6 and 2026/03312-3; by the Conselho Nacional de Desenvolvimento Científico e Tecnológico (CNPq), grant 301878/2025-0; and by the Coordenação de Aperfeiçoamento de Pessoal de Nível Superior - Brasil (CAPES), through the MATH-AmSud program, grant 88881.179491/ 2025-01.
PS was supported by S\~ao Paulo Research Foundation (FAPESP), grants 2021/01799-9 and 2024/15612-6. DJWS was supported by Marsden Fund contract MAU2504 managed by Royal Society Te Ap\={a}rangi.

\end{document}